\documentclass{amsart}
\usepackage{amssymb}
\usepackage{latexsym}

\date{\today}

\newcommand{\Z}{{\mathbb Z}}
\newcommand{\R}{{\mathbb R}}
\newcommand{\Q}{{\mathbb Q}}
\newcommand{\C}{{\mathbb C}}
\newcommand{\D}{{\mathbb D}}
\newcommand{\T}{{\mathbb T}}

\newcommand{\PP}{{\mathbb P}}
\newcommand{\HH}{{\mathbb H}}

\newcommand{\CD}{{\mathcal D}}
\newcommand{\CE}{{\mathcal E}}

\newcommand{\CB}{{\mathcal B}}
\newcommand{\CH}{{\mathcal H}}
\newcommand{\CI}{{\mathcal I}}

\newcommand{\CU}{{\mathcal U}}
\newcommand{\CK}{{\mathcal K}}

\newtheorem{theorem}{Theorem}
\newtheorem{lemma}{Lemma}
\newtheorem{defi}{Definition}
\newtheorem{prop}{Proposition}
\newtheorem*{subcon}{Subshift Conjecture}
\newtheorem{corollary}{Corollary}
\begin{document}

\title[Singular Density of States Measure]{Singular Density of States Measure for Subshift and Quasi-Periodic Schr\"{o}dinger Operators}

\author[A.\ Avila]{Artur Avila}

\address{
CNRS UMR 7586, Institut de Math\'ematiques de Jussieu - Paris Rive Gauche,
B\^atiment Sophie Germain, Case 7012, 75205 Paris Cedex 13, France
\&
IMPA, Estrada Dona Castorina 110, 22460-320, Rio de Janeiro, Brazil}

\email{{artur@math.jussieu.fr}}

\urladdr{http://w3.impa.br/$\sim$avila/}

\thanks{A.\ A.\ was supported by the ERC Starting Grant ``Quasiperiodic'' and by the Balzan project of Jacob Palis.}

\author[D.\ Damanik]{David Damanik}

\address{Department of Mathematics, Rice University, Houston, TX~77005, USA}

\email{damanik@rice.edu}

\urladdr{www.ruf.rice.edu/$\sim$dtd3}

\thanks{D.\ D.\ was supported in part by a Simons Fellowship and NSF grants DMS--0800100 and DMS--1067988.}

\author[Z.\ Zhang]{Zhenghe Zhang}

\address{Department of Mathematics, Northwestern University, Evanston, IL~60208, USA}

\email{zhenghe@math.northwestern.edu}

\begin{abstract}
Simon's subshift conjecture states that for every aperiodic minimal subshift of Verblunsky coefficients, the common essential support of the associated measures has zero Lebesgue measure. We disprove this conjecture in this paper, both in the form stated and in the analogous formulation of it for discrete Schr\"odinger operators. In addition we prove a weak version of the conjecture in the Schr\"odinger setting. Namely, under some additional assumptions on the subshift, we show that the density of states measure, a natural measure associated with the operator family and whose topological support is equal to the spectrum, is singular. We also consider one-frequency quasi-periodic Schr\"odinger operators with continuous sampling functions and show that generically, the density of states measure is singular as well.
\end{abstract}

\maketitle

\section{Introduction}\label{sec.1}

The theme of this paper is driven by the desire to prove zero-measure spectrum in several instances. One scenario in which zero-measure spectrum has been shown to be typical was reviewed in \cite{D}. If one considers an aperiodic strictly ergodic subshift over a finite alphabet, which is assumed to consist of real numbers for simplicity, and considers Schr\"odinger operators in $\ell^2(\Z)$ with potentials given by the elements of the subshift, then the (by minimality) common spectrum of these operators has a strong tendency to be a set of zero Lebesgue measure. In fact, it has been shown that the so-called Boshernitzan condition, which holds for many strictly ergodic subshifts \cite{DL2}, is a sufficient criterion for zero-measure spectrum \cite{DL1}. Moreover, there are aperiodic strictly ergodic subshifts for which the Boshernitzan condition fails but zero-measure spectrum holds \cite{LQ}.

Thus, one is tempted to conjecture that zero-measure spectrum is indeed a universal feature of this class of models. Indeed, Barry Simon conjectured in \cite{S2} that minimality and aperiodicity should be sufficient.\footnote{To be precise, the conjecture as formulated in \cite{S2} makes this statement for the unitary analogues, the so-called CMV matrices. We will comment on this distinction below.} We will show that the conjecture is false. Indeed, we will construct an aperiodic minimal subshift for which the associated spectrum has positive Lebesgue measure.

A weakening of zero-measure spectrum that is still interesting in its own right is asking for singularity of the density of states measure. This measure depends on the choice of an ergodic measure $\mu$ on the subshift and it is associated with the operator family in a natural way. Namely, it is simply given by the average with respect to the chosen ergodic measure of the spectral measure corresponding to the operator and the delta function at the origin. The topological support of this measure is equal to the ($\mu$-almost sure) spectrum of the operators and hence zero-measure spectrum implies the singularity of the density of states measure. On the other hand, one may well ask whether there is \emph{some} support of the measure that has zero Lebesgue measure even in cases where zero-measure spectrum fails to hold.

We will identify a sufficient condition for the singularity of the density of states measure associated with a subshift and an ergodic measure on it. This condition will be formulated in terms of polynomial transitivity and polynomial factor complexity properties of the subshift that hold almost surely with respect to the measure. We will give several examples to which this result can be applied. These will include subshifts generated by codings of shifts and skew-shifts on tori, as well as codings of interval exchange transformations.

The second scenario in which one desires to prove zero-measure spectrum is the class of one-frequency quasi-periodic Schr\"odinger operators in $\ell^2(\Z)$. The potentials of these operators are generated by an irrational rotation of the circle and a continuous sampling function. Holding the rotation fixed and varying the sampling function, one may ask about the typical spectral behavior in the sense of Baire. That is, is there a spectral phenomenon that holds for a residual set of continuous sampling functions? It has been shown that in this sense, the absence of absolutely continuous spectrum \cite{AD} as well as the absence of point spectrum \cite{BD} are typical. Thus, for fixed irrational rotation, there is a residual set of sampling functions such that for each of them, the operators have purely singular continuous spectrum (for Lebesgue almost every point on the circle). Thus, the typical spectral type in this setting is the same as that of operators with subshift potentials.

The analogy here is even closer than it appears at first sight. The results in \cite{AD, BD} are proved by approximating a continuous sampling function by step functions in the uniform topology. The latter induce potentials that may be considered as subshift potentials and their spectral properties are pushed through to the limit. Thus, the known results suggest that what is typical for subshift potentials should also be typical for a generic continuous sampling function.

Since the approximating step functions can always be chosen so that the associated operators have zero-measure spectrum (as shown in \cite{DL2}), one may conjecture that for generic continuous sampling functions, one should have zero-measure spectrum as well. We are neither able to prove this, nor are we able to disprove this. However, in line with our results in the subshift setting, we are able to show that the density of states measure is singular for generic continuous sampling functions. We also show that there are sampling functions for which the density of states measure is singular, but for which the spectrum does have positive measure. That is, zero-measure spectrum is indeed a strictly stronger property in the context of one-frequency quasi-periodic Schr\"odinger operators. In fact, in the examples we construct to exhibit this phenomenon, the spectrum actually contains an interval. This result is of additional independent interest since this also provides the first example of a one-frequency quasi-periodic Schr\"odinger operator whose spectrum is not a Cantor set.

The structure of the paper is as follows. In Section~\ref{sec.2} we collect a few general results that apply to (and will be used in) both scenarios we consider. In Section~\ref{sec.3} we discuss the subshift setting, disprove the subshift conjecture, prove our weak replacement of it, and apply the latter result to several examples. Finally, in Section~\ref{sec.4}, we consider the one-frequency quasi-periodic setting, prove generic singularity of the density of states measure, and construct examples with singular density of states measure and spectrum containing an interval.

\bigskip

\noindent\textit{Acknowledgments.} D.~D.\ and Z.~Z.\ would like to thank IMPA for the kind hospitality and financial support for a stay at the institute during which much of this work was done. Z.~Z.'s travel to IMPA was supported by NSF grant DMS-1001727 (PI: A.~Wilkinson).

\section{Preliminaries}\label{sec.2}

Since we are mainly interested in Schr\"odinger operators in this paper, and all the potentials we consider fit into the framework of dynamically defined potentials, let us recall the framework and some general results that hold in the general case. All statements in this subsection are well known. We refer the reader to \cite{CL, CFKS, J, K2, PF} for proofs and further background.

Given a compact metric space $\Omega$, a homeomorphism $T : \Omega \to \Omega$, and a bounded Borel measurable sampling function $f : \Omega \to \R$, we define the family $\{ H_\omega \}_{\omega \in \Omega}$ of Schr\"odinger operators in $\ell^2(\Z)$ by
\begin{equation}\label{e.oper}
[H_\omega \psi](n) = \psi(n+1) + \psi(n-1) + V_\omega(n) \psi(n),
\end{equation}
with the potentials $V_\omega(n) = f(T^n \omega)$, $\omega \in \Omega$, $n \in \Z$.

Suppose $\mu$ is a $T$-ergodic probability measure on $\Omega$. The associated density of states measure $dk$ on $\R$ is given by
$$
\int g \, dk = \int_\Omega \langle \delta_0 , g(H_\omega) \delta_0 \rangle \, d\mu(\omega).
$$
Its distribution function,
$$
k(E) = \int \chi_{(-\infty,E]} \, dk,
$$
is called the integrated density of states (IDS). Both depend on $T$, $f$, and $\mu$, but the dependence will usually be suppressed from the notation. It is a standard result that the integrated density of states is always continuous in our setting. In other words, the density of states measure has no atoms.

\begin{prop}
Given a $T$-ergodic probability measure $\mu$ on $\Omega$, there is a set $\Sigma_\mu \subset \R$ such that $\sigma(H_\omega) = \Sigma_\mu$ for every $\mu$-almost every $\omega \in \Omega$. Moreover, $\Sigma_\mu$ is the topological support of $dk_\mu$.
\end{prop}

In general, the set $\Sigma_\mu$ depends on $\mu$. However, we have the following simple criterion for $\sigma(H_\omega)$ to be entirely independent of $\omega$, which is an easy consequence of strong convergence.

\begin{prop}\label{p.minspec}
If $T$ is minimal and $f$ is continuous, then there is a compact set $\Sigma \subset \R$ such that $\sigma(H_\omega) = \Sigma$ for every $\omega \in \Omega$.
\end{prop}

Given a continuous function $h : \Omega \to \R$, we consider the $\mathrm{SL}(2,\R)$ cocycle $(T, A^h)$ over $T$, given by
$$
(T, A^h) : \Omega \times \R^2 \to \Omega \times \R^2, \quad (\omega, v) \mapsto (T \omega , A^h(\omega) v),
$$
where
$$
A^h(\omega) = \begin{pmatrix} h(\omega) & - 1 \\ 1 & 0 \end{pmatrix}.
$$
Note that for every $n \in \Z$, we have $(T, A^h)^n = (T^n, A^h_n)$ with a suitable function $A^h_n : \Omega \to \mathrm{SL}(2,\R)$. This cocycle is called \emph{uniformly hyperbolic} if $\|A^h_n(\omega)\| \ge c \lambda^{|n|}$ for suitable $c > 0$ and $\lambda > 1$, uniformly in $\omega \in \Omega$. Uniform hyperbolicity has various equivalent descriptions, for example in terms of exponential dichotomy in the sense of \cite{J}; see \cite{Y} for further discussion.

If we consider a function of the form $h(\omega) = E - f(\omega)$ with a continuous sampling function $f$ and some energy $E \in \R$, then the associated cocycle $(T, A^{E - f})$ generates the standard transfer matrices associated with the difference equation $H_\omega u = E u$.

Here is an important result of Johnson \cite[Theorem~3.1]{J}:

\begin{prop}\label{p.johnson}
Let $f \in C(\Omega,T)$. Suppose $\omega \in \Omega$ has a dense $T$-orbit. Then, an energy $E \in \R$ belongs to the resolvent set of $H_\omega$ {\rm (}i.e., $E \not\in \sigma(H_\omega)${\rm )} if and only if the cocycle $(T, A^{E - f})$ is uniformly hyperbolic.

In particular, if $T$ is minimal, then
$$
\R \setminus \Sigma = \CU\CH,
$$
where $\CU\CH := \{ E \in \R : (T, A^{E - f}) \text{ is uniformly hyperbolic} \}$.
\end{prop}

By the subadditive ergodic theorem, for each $E \in \R$, there is $L(E) \ge 0$, called the Lyapunov exponent at energy $E$, and a set $\Omega_E$ of full measure such that for every $\omega \in \Omega_E$, we have
$$
\lim_{n \to \infty} \frac1n \log \| A^{E-f}_n (\omega) \| = L(E).
$$
We have the following important result of Kotani \cite{K2}, which will be crucial in what we do in the subshift setting.

\begin{prop}\label{p.kotani}
Suppose $f$ takes finitely many values. Then, we have the following dichotomy: either $V_\omega$ is periodic for $\mu$-almost every $\omega$, or $L(E) > 0$ for Lebesgue almost every $E \in \R$.
\end{prop}

\section{Simon's Subshift Conjecture}\label{sec.3}

\subsection{Statement of the Conjecture and Discussion}

Here is the subshift conjecture of Barry Simon (which is \cite[Conjecture~12.8.2]{S2}):

\begin{subcon}
Given a minimal subshift of Verblunsky coefficients which is not periodic, the common essential support of the associated measures has zero Lebesgue measure.
\end{subcon}

Let us clarify the statement. Given any one-sided infinite sequence $\{ \alpha_n \}_{n \ge 0} \subset \D = \{ z \in \C : |z| < 1 \}$, there is an associated probability measure on the unit circle $\partial \D = \{ z \in \C : |z| = 1 \}$ in such a way that the $\alpha_n$'s are the canonical recursion coefficients of the orthogonal polynomials associated with the measure; see \cite{S1} for details. The $\alpha_n$'s are called the Verblunsky coefficients. Consider the case where the Verblunsky coefficients take values in a fixed finite subset $\mathcal{A}$ of $\D$. In this case, the sequence may be regarded as an element of the compact space $\mathcal{A}^{\Z_+}$ (equipped with product topology). The shift transformation $T : \mathcal{A}^{\Z_+} \to \mathcal{A}^{\Z_+}$ is given by $(T \alpha)_n = \alpha_{n+1}$. A $T$-invariant closed subset $\Omega$ of $\mathcal{A}^{\Z_+}$ is called a subshift  of Verblunsky coefficients. The subshift $\Omega$ is called minimal if the $T$-orbit of every $\alpha \in \Omega$ is dense in $\Omega$. If $\Omega$ is minimal, then either all elements are periodic or all elements are not periodic. In the latter case, we say the subshift is not periodic. Suppose we are given a minimal subshift of Verblunsky coefficients which is not periodic. Then, by \cite[Theorem~12.8.1]{S2}, there is a set $\Sigma \subseteq \partial \D$ such that for each $\alpha \in \Omega$, the topological support of the associated measure minus its isolated points is given by $\Sigma$. The subshift conjecture asserts that $\Sigma$ must have zero Lebesgue measure.

We will disprove the subshift conjecture. In fact, we will shift the setting to the Schr\"odinger operator context for the time being. The reason for doing this is twofold. First, the evidence for the subshift conjecture, at the time \cite{S2} was written, was essentially purely on the Schr\"odinger side and hence it was really the OPUC version of a conjecture that one would want to make in the Schr\"odinger context. (Later it was shown in \cite{DL3} that similar supporting evidence may be obtained in the setting of Verblunsky coefficients.) Second, since the present paper primarily discusses Schr\"odinger operators, the presentation of the proof is somewhat more natural in that setting. We will, however, discuss the modifications necessary when switching to the OPUC setting. In particular, it does turn out that \cite[Conjecture~12.8.2]{S2} as stated is false.

\subsection{Disproof of the Subshift Conjecture in the Schr\"odinger Case}

Let us first state the Schr\"odinger version of the subshift conjecture. Suppose that $\mathcal{A}$ is a finite subset of $\R$. Define the shift transformation $T$ of the compact space $\mathcal{A}^\Z$ as above and consider a $T$-invariant compact subset $\Omega$ of $\mathcal{A}^\Z$. (Note that we are now considering two-sided sequences of real numbers as opposed to one-sided infinite sequences of elements of $\D$.) Given such a subshift $\Omega$, we consider the familiy $\{ H_\omega \}_{\omega \in \Omega}$ of Schr\"odinger operators given by
$$
[H_\omega \psi](n) = \psi(n+1) + \psi(n-1) + \omega_n \psi(n).
$$
That is, the operators are as in \eqref{e.oper}, where the sampling function is the evaluation at the origin, $f(\omega) = \omega_0$. If $\Omega$ is minimal, then by Proposition~\ref{p.minspec} there is a set $\Sigma \subset \R$ such that $\sigma(H_\omega) = \Sigma$ for every $\omega \in \Omega$. We can now state the announced version of the conjecture.

\begin{subcon}[Schr\"odinger Version]
Given $\mathcal{A} \subset \R$ finite and a minimal subshift $\Omega \subset \mathcal{A}^\Z$ which is not periodic, the associated set $\Sigma$ has zero Lebesgue measure.
\end{subcon}

Let us now show that the Schr\"odinger version of the subshift conjecture fails.

\begin{theorem}\label{t.3}
Given $\mathcal{A} \subset \R$ with $2 \le \mathrm{card} \, \mathcal{A} < \infty$, there is a minimal subshift $\Omega \subset \mathcal{A}^\Z$, which is not periodic, such that the associated set $\Sigma \subset \R$ has strictly positive Lebesgue measure.
\end{theorem}

\begin{proof}
We first introduce some notation. For a finite word $w$ over $\mathcal{A}$ of length $n$ and an energy $E$, let us write the monodromy matrix over $w$ corresponding to energy $E$ as $A^{E,w}_n$. Explicitly,
$$
A^{E,w}_n = \begin{pmatrix} E - w_n & - 1 \\ 1 & 0 \end{pmatrix} \cdots \begin{pmatrix} E - w_1 & - 1 \\ 1 & 0 \end{pmatrix}.
$$
The corresponding periodic spectrum is given by
$$
\Sigma(w) = \left\{ E : | \mathrm{Tr} A^{E,w}_n | \le 2 \right\}.
$$

The construction will be iterative. Fix an integer $k_1 \ge 2$. We choose $k_1$ words
$$
w_{1,1}, \ldots, w_{1,k_1}
$$
over the alphabet $\mathcal{A}$. We associate with them $k_1$ periodic potentials, where $w_{1,j}$ just corresponds to the periodic block of the $j$-th potential. The associated subshift $\Omega_1$ at this stage consists of all two-sided infinite concatenations of the $w_{1,j}$. The only requirement at this stage is that $\Omega_1$ contains non-periodic sequences. This is clearly possible since both the alphabet size and $k_1$ are at least $2$. The associated spectrum at this stage is the union of the $k_1$ periodic spectra,
$$
\Sigma_1 = \bigcup_{j = 1}^{k_1} \Sigma(w_{1,j}).
$$

Now we pass to the second step. We form the word $W_1 = w_{1,1} \cdots w_{1,k_1}$. For each $k \in \{ 1, \ldots, k_1 \}$, we choose an appropriate power $m_{1,k} \ge 2$ (we explain below how to choose it) and form the second step words
$$
W_1 w_{1,k}^{s}, \quad 1 \le s \le m_{1,k} , \; 1 \le k \le k_1.
$$
The new words may be listed as
$$
w_{2,1},\ldots,w_{2,k_2},
$$
they again correspond to periodic potentials, and there is an associated subshift $\Omega_2$ and an associated spectrum $\Sigma_2$, which are defined in a way analogous to step one. Namely, $\Omega_2$ consists of all two-sided infinite concatenations of the $w_{2,k}$ and $\Sigma_2$ is the union of the $k_2$ periodic spectra $\Sigma(w_{2,k})$, $1 \le k \le k_2$.

Now repeat this procedure and obtain $\Omega_\ell$ and $\Sigma_\ell$. Clearly, the subshifts are decreasing, $\Omega_\ell \supseteq \Omega_{\ell+1}$. While the spectra $\Sigma_\ell$ are not necessarily decreasing, the powers $m_{\ell,k}$ in the various steps will be chosen so that
\begin{equation}\label{e.perspecest}
\mathrm{Leb} (\Sigma_\ell \setminus \Sigma_{\ell+1}) < \mathrm{Leb} (\Sigma_1) 2^{-(1+\ell)}.
\end{equation}
To ensure this estimate for appropriate choices of the $m_{\ell,k} \ge 2$, it suffices to prove the following lemma.

\begin{lemma}\label{l.6}
For any pair $v, w$ of finite words, we have
$$
\lim_{m \to \infty} \mathrm{Leb} \left( \Sigma(w) \setminus \bigcup_{k = 1}^m \Sigma(vw^k) \right) = 0.
$$
\end{lemma}

\begin{proof}
It is well known that for a finite word $p$ of length $l$, $\Sigma(p)$ can be written as a union of $l$ compact non-degenerate intervals, called the bands, whose interiors are mutually disjoint.\footnote{Explicitly, the set $\{ E : | \mathrm{Tr} A^{E,p}_l | < 2 \}$ has exactly $l$ connected components and the closure of each of them gives rise to a band.}

Assume $v$ has length $n$ and $w$ has length $\ell$. Consider energies $E$ in the interior of a band of $\Sigma(w)$. Then the following facts are well known.
\begin{itemize}

\item For each $E$ in the interior of $\Sigma(w)$, we have for some $P(E) \in \mathrm{SL}(2,\R)$ and $\theta(E)$ such that
$$
A_\ell^{E,w} = P(E) R_{\theta(E)} P(E)^{-1},
$$
where $R_\theta$ is the rotation matrix of rotation angle $2 \pi \theta$;

\item $\theta(E)$ is smooth and strictly monotone for $E$ in the interior of $\Sigma(w)$.

\end{itemize}

Thus, for Lebesgue almost every $E$ in the interior of $\Sigma(w)$, $\theta(E)$ is irrational. Fix such an energy $E$, let $\theta = \theta(E)$, $P = P(E)$, $A = A^{E,v}_n$ and $B = A^{E,w}_\ell$. We consider
$$
B^k A = P R_{k \theta} P^{-1} A P P^{-1}.
$$
By polar decomposition, we have for some $\alpha$ and $\beta$,
$$
P^{-1} A P = R_\alpha \begin{pmatrix} \|P^{-1} A P\| & 0 \\ 0 & \|P^{-1} A P\|^{-1} \end{pmatrix} R_\beta.
$$
Now it is straightforward that the composition will be elliptic if $\|\alpha+ \beta + k \theta\|_{\R/\Z}$ is sufficiently close to $\frac{1}{4}$. Since $\{k \theta\}_{k \ge 1}$ is dense in $\R/\Z$, we can find such a $k$. Thus, for Lebesgue almost every $E$ in $\Sigma(w)$, there exists $k \ge 1$ such that $E \in \Sigma(vw^k)$. This implies
$$
\lim_{m \to \infty} \mathrm{Leb} \left( \Sigma(w) \setminus \bigcup_{k = 1}^m \Sigma(vw^k) \right) = 0,
$$
as claimed.
\end{proof}

It follows that we can indeed ensure that \eqref{e.perspecest} holds. Now we consider the limit subshift
$$
\Omega = \lim_{\ell \to \infty} \Omega_\ell = \bigcap_{\ell \ge 1} \Omega_\ell.
$$
Clearly, $\Omega$ is a subshift. Furthermore, we have the following:

\begin{lemma}\label{l.7}
The subshift $\Omega$ is minimal and non-periodic.
\end{lemma}

\begin{proof}
It is a standard result that minimality in the subshift context is equivalent to the fact that every finite word, which occurs in some element of $\Omega$, occurs in every element of $\Omega$ infinitely often with bounded gaps between consecutive occurrences; see, for example, \cite{Q}. Now recall that
$$
w_{1,j},\ 1 \le j \le k_1
$$
are our first steps words. Then by induction we have
\begin{equation}\label{e.words1}
W_\ell = w_{\ell,1} \cdots w_{\ell,k_\ell}
\end{equation}
and
\begin{equation}\label{e.words2}
w_{\ell+1, k'} = W_\ell w_{\ell,k}^s,\ 1 \le s \le m_{\ell,k},\ 1 \le k \le k_\ell.
\end{equation}
We claim the following facts:
\begin{itemize}

\item $W_\ell$ occurs with bounded gaps in $\Omega_{\ell+1}$ and $n_\ell = |W_\ell| \to \infty$ as $\ell \to \infty$.

\item $W_\ell$ contains all possible finite words of length less than or equal to $n_{\ell-2}$ that occur in elements of $\Omega_\ell$.

\end{itemize}
It is clear that these two facts imply minimality since
$$
\Omega = \lim_{\ell \to \infty} \Omega_\ell = \bigcap_{\ell \ge 1} \Omega_\ell.
$$
For the proof of the two facts, the first is straightforward by construction. For the second one, we only need to note that all elements in $\Omega_\ell$ are of the form
$$
\cdots W_{\ell-1} w_{\ell-1,k}^s W_{\ell-1} w_{\ell-1,k'}^{s'} W_{\ell-1} \cdots
$$
and each $w_{\ell-l,k}$ contains $W_{\ell-2}$. Thus, all possible words of length less than or equal to $n_{\ell-2}$ in each element of $\Omega_\ell$ are contained in
$$
W_{\ell-1} w_{\ell-1,k}^{s} W_{\ell-1}, \quad 1 \le k \le k_{\ell-1},\ 1 \le s \le m_{\ell-1,k}.
$$
By \eqref{e.words1} and \eqref{e.words2}, $W_\ell$ contains all words of the above form, except for $W_{\ell-1} w_{\ell-1,k_{\ell-1}}^{m_{\ell-1,k_{\ell-1}}} W_{\ell-1}$. But since $m_{\ell-1,k_{\ell-1}} \ge 2$, it is still true that all words of length less than or equal to $n_{\ell-2}$ occurring in $W_{\ell-1} w_{\ell-1,k_{\ell-1}}^{m_{\ell-1,k_{\ell-1}}} W_{\ell-1}$ already occur in $W_{\ell-1} w_{\ell-1,k_{\ell-1}}^{m_{\ell-1,k_{\ell-1}}-1} W_{\ell-1}$, which in turn does occur in $W_\ell$. Thus, minimality of $\Omega$ follows.

To show that $\Omega$ is not periodic, it suffices to show that for each $l_0 < \infty$, there is a length $l \ge l_0$ such that $\Omega$ contains a word $w$ of length $l$ that has at least two right-extensions $wa, wb$, $a \not= b$, both occurring in $\Omega$ (again, compare \cite{Q}). The latter property is obvious from the construction and our initial choice of the first step words. This completes the proof of Lemma~\ref{l.5}.
\end{proof}

Now by minimality, the spectrum of the operators $H_{\omega}$ is independent of $\omega \in \Omega$ and hence may be denoted by $\Sigma$. To get the estimate $\mathrm{Leb} (\Sigma) > 0$, we need the following lemma.

\begin{lemma}\label{l.8}
Let $\Sigma_\ell$ and $\Sigma$ be as in the construction. Assume $m_{\ell,k}\ge 2$ for all $\ell \ge 1$ and $k \ge 1$. Then,
$$
\Sigma \supseteq \limsup_{\ell \to \infty} \Sigma_\ell = \bigcap_{L \ge 1} \bigcup_{\ell \ge L} \Sigma_\ell.
$$
\end{lemma}

\begin{proof}
Fix $E \notin \Sigma$ and consider the sampling function $f : \Omega \to \R$, $f(\omega) = \omega_0$. Then, by minimality of $\Omega$, the cocycle $(T,A^{E-f}) : \Omega \times \R^2 \to \Omega \times \R^2$ is uniformly hyperbolic; compare Proposition~\ref{p.johnson}. Hence there exist continuous maps $B : \Omega \to \mathrm{PSL}(2,\R)$ and $r : \Omega \to \R^+$ such that
$$
A^{E-f}(\omega) = B(T \omega)^{-1} \begin{pmatrix} r(\omega) & 0 \\ 0 & r^{-1}(\omega) \end{pmatrix} B(\omega).
$$
Furthermore, $r$ can be chosen such that
$$
r_n(\omega) = \prod^{n-1}_{k=0} r(T^k \omega) > c \lambda^n,
$$
for some $c > 0$ and $\lambda > 1$. Thus there exists a constant $C$ such that for all $\omega \in \Omega$ and all $n \ge 1$,
$$
C^{-1} r_n(\omega) < \|A^{E-f}_n(\omega)\| < C r_n(\omega).
$$
It is easy to see, by polar decomposition, that for any $A \in \mathrm{SL}(2,\R)$ and $\delta > 0$, if $\|A\|$ is sufficiently large, then $\|A^2\| > C \|A\|^{1+\delta}$ implies that $A$ is hyperbolic.

Now by construction, if we choose $m_{\ell, k}\ge 2$, then all finite words $w_{\ell, k}^2$ occur in elements of $\Omega$. Also we note that $|w_{\ell,k}| \to \infty$ as $\ell \to \infty$, uniformly in $k$.

Now for simplicity, let $w = w_{\ell,k}$ and $n = |w|$, and pick $\omega \in \Omega$ such that $w^2$ is the finite word that occurs in the $[0, \ldots, 2n-1]$ position of $\omega$. Then there exists $L \in \Z^+$ such that for all $\ell > L$, $\|A^{E,w}_n\|$ is sufficiently large and
\begin{align*}
\|(A^{E,w}_n)^2\| & = \|A^{(E-f)}_{2n}(\omega)\| \\
& > C^{-1} r_{2n}(\omega) \\
& = C^{-1} r_n(T^n \omega) \, r_n(\omega) \\
& > C^{-3} \|A^{(E-f)}_n (T^n \omega)\| \cdot \|A^{(E-f)}_n(\omega)\| \\
& = C^{-3} \|A^{E,w}_n\|^2.
\end{align*}
Thus, $A^{E,w}_n$ is hyperbolic. This implies that $E \not\in \Sigma_\ell$ for all $\ell > L$. In particular, it follows that
$$
E \notin \bigcap_{L \ge 1} \bigcup_{\ell \ge L} \Sigma_\ell,
$$
which completes the proof of the lemma.
\end{proof}

We can now conclude the proof of Theorem~\ref{t.3}. By Lemma~\ref{l.7}, the subshift $\Omega$ we constructed is minimal and not periodic. From Lemma~\ref{l.8} we may infer that
$$
\Sigma \supseteq \limsup_{\ell \to \infty} \Sigma_\ell = \bigcap_{L \ge 1} \bigcup_{\ell \ge L} \Sigma_\ell \supseteq \left( \left( \Sigma_1 \setminus (\Sigma_1 \setminus \Sigma_2) \right) \setminus (\Sigma_2 \setminus \Sigma_3) \right) \setminus \cdots.
$$
Thus, by estimate \eqref{e.perspecest}, we must have $\mathrm{Leb} (\Sigma) \ge \frac{1}{2} \mathrm{Leb} (\Sigma_1) > 0$.
\end{proof}

\subsection{Disproof of the Subshift Conjecture in the OPUC Case}

Here is the OPUC analog of Theorem~\ref{t.3}, which disproves the subshift conjecture in its original formulation:

\begin{theorem}\label{t.3b}
Given $\mathcal{A} \subset \D$ with $2 \le \mathrm{card} \, \mathcal{A} < \infty$, there is a minimal subshift $\Omega \subset \mathcal{A}^{\Z_+}$, which is not periodic, such that the associated set $\Sigma \subset \partial \D$ has strictly positive Lebesgue measure.
\end{theorem}

\begin{proof}
The proof is analogous to the proof of Theorem~\ref{t.3}. Let us discuss the necessary adjustments. The Floquet theory underlying the proof of Lemma~\ref{l.6} has an OPUC analog, compare \cite[Section~11.2]{S2}, and using it, the OPUC analog of the lemma may be established. The characterization of the complement of $\Sigma$ in terms of uniform hyperbolicity underlying the proof of Lemma~\ref{l.8} has an OPUC version as well, see \cite{DL3}, and as a consequence, this lemma carries over also. Given these two ingredients, one can construct the desired subshift in the exact same way and prove that it is minimal and non-periodic (note that the proof of Lemma~\ref{l.7} is purely combinatorial and applies to both scenarios), and that the associated set $\Sigma \subset \partial \D$ has strictly positive Lebesgue measure.
\end{proof}

\subsection{A Weaker Positive Result}

After disproving the subshift conjecture above by constructing minimal aperiodic subshifts with positive-measure spectrum, we pursue in this subsection the more modest goal of showing that there is at least some set of zero Lebesgue measure that supports the density of states measure. Note that the density of states measure will depend on the choice of the ergodic measure, while the spectrum was independent of this choice for a given minimal subshift. Thus, the conditions we will impose here will naturally involve both the subshift and the ergodic measure. The sufficient conditions will be contained in Definitions~\ref{d.1} and \ref{d.2} below and the result on the singularity of the density of states measure will be given in Theorem~\ref{t.4} below.

\bigskip

Let us first give an informal description of the idea behind the proof. For simplicity, let us look at the half-line case and denote the boundary condition by $\theta$. Consider the sets
$$
S_{\omega,C,\theta,n} = \{ E \in \R : \|A^{E-f}_k(\omega) u_\theta\| \le C (1 + k),  \; 1 \le k \le n\},
$$
where $u_\theta$ is the initial vector that satisfies the boundary condition, and
$$
S_{\omega,C,\theta} = \bigcap_{n \ge 1} S_{\omega,C,\theta,n}.
$$
Then, by the standard result on the existence of generalized eigenfunctions and the definition of the density of states measure, $\bigcup_C S_{\omega,C,\theta}$ supports the spectral measure of $H_{\omega,\theta}$ and, for every support $\Omega'$ of the fixed ergodic measure $\mu$, $\bigcup_{\omega \in \Omega'} \bigcup_C S_{\omega,C,\theta}$ supports the $\mu$-average of these measures, which we want to prove to be singular.

Thus, our goal is to find a support $\Omega'$ of $\mu$ such that for every $C > 0$,
$$
\mathrm{Leb} \left( \bigcup_{\omega \in \Omega'} S_{\omega,C,\theta} \right) = 0.
$$

For this, it suffices to show that
$$
\lim_{n \to \infty} \mathrm{Leb} \left( \bigcup_{\omega \in \Omega'} S_{\omega,C,\theta,n} \right) = 0.
$$

Notice that $S_{\omega,C,\theta,n}$ only depends on $\omega_1,\ldots,\omega_n$, of which we will assume there are only polynomially many possibilities, say $p(n) \le n^\gamma$ with $\gamma < \infty$. If we can prove that $\mathrm{Leb} \left( S_{\omega,C,\theta,n} \right)$ is correspondingly small, uniformly in $\omega$, we are therefore done. That is, we need
$$
\sup_{\omega \in \Omega'} \mathrm{Leb} \left( S_{\omega,C,\theta,n} \right) = o(n^{-\gamma}).
$$

Now, assuming aperiodicity, use that by Kotani (cf.~Proposition~\ref{p.kotani}), $L(E) > 0$ for Lebesgue almost every $E$. Show that as $E$ varies, the most contracted direction of $A^{E-f}_k(\omega)$ moves with a velocity that is bounded away from zero. Thus, for the energies $E$ with $A^{E-f}_k(\omega)$ super-polynomially large, only a super-polynomially small set of those energies will have $\|A^{E-f}_k(\omega) u_\theta\| \le C (1 + k)$ (since $\theta$ is fixed and serves as a target that needs to be close to the most contracted direction of $A^{E-f}_k(\omega)$). This proves what we need and establishes the existence of a zero measure set that supports the $\mu$-average of the spectral measures.

\bigskip

Now let us turn to the whole-line case at hand. We again use the description of the support in terms of generalized eigenfunctions. But this time, we use that $A^{E-f}_k(\omega)$ is large for both negative and positive $k$, and as $E$ varies, both the negative and the positive half-line have the direction of most contraction moving with velocity bounded away from zero, and they move in opposite directions! Thus, again, a match sufficient to ensure the linear estimate can only occur on a super-polynomially small set of energies since the norms are super-polynomially large.

\bigskip

We formulate this in the following lemma:

\begin{lemma}\label{l.9}
Given $\eta > 0$, $M > 0$, and a polynomial $p(x) \in \R[X]$, there exists $N_0$ such that for every $n \geq N_0$, the following holds. If we consider Schr\"odinger matrices
$$
A^E(l) = A^{E-v_l} = \begin{pmatrix} E - v_l & - 1 \\ 1 & 0 \end{pmatrix},
$$
where $-M < v_l < M$, $l \in [-n,n-1]$, the products
$$
A^E_m = \begin{cases} A^E(m-1) \cdots A^E(0) & m \ge 1, \\ I & m = 0, \\ (A^E(m))^{-1} \cdots (A^E(-1))^{-1} & m \le -1, \end{cases}
$$
and the set $\CE$ consisting of those energies $E \in \R$ satisfying the following two conditions,
\begin{itemize}

\item there exist $k_1, k_2 > 0$ such that
$$
\|A^E_{k_1}\| > e^{\eta n},\ \|A^E_{-k_2}\| > e^{\eta n};
$$

\item there exists a unit vector $w = w(E) \in \R^2$ with $\|A^E_l w\| < p(n)$ for all $l \in [-n,n]$,

\end{itemize}
then $\mathrm{Leb}(\CE) < C n^3 p(n) e^{-\eta n}$ for some universal constant $C>0$.
\end{lemma}

\begin{proof}
It is easy to see that for $E \notin [-M-2, M+2]$, there is an invariant cone field for the sequence $A^E(l),\ l=-n,\ldots,n-1$, inside which the vectors are expanded by $A^E(l)$ for each $l$. Thus for large $n$, we can restrict to the interval $[-M-2,M+2]$.

Given $k_1, k_2 \in [1,n]$, \cite[Lemma~2.4]{A} implies that if we consider the set $\CD \subset \R$ of those energies $E \in  \R$ for which we have
\begin{itemize}

\item $\|A^E_{k_1}\|\geq e^{\eta n}$ and $\|A^E_{-k_2}\|\geq e^{\eta n}$,

\item at least one of $A^E_{k_1}, A^E_{-k_2}$ is not hyperbolic,

\end{itemize}
then $Leb(\CD) \leq 8 \pi ne^{-\eta n}$. Considering all possible pairs $(k_1,k_2)$, we get a set of measure at most $8 \pi n^3e^{-\eta n}$. Thus we only need to estimate the measure of $\CE \setminus \CD$. In other words, we may assume in addition that both $A^E_{k_1}$ and $A^E_{-k_2}$ are hyperbolic.

Denote the angle of most contracted direction of $A\in\mathrm{SL}(2,\R)$ by $s(A)\in\R\PP^1=\R/(\pi\Z)$; let $s_k(E)=s(A^E_k)$ and $u_{k}(E) = s(A^E_{-k})$. Then it is easy to see that for $n$ large enough and $E \in \CE$, we have
$$
|s_{k_1}(E) - u_{k_2}(E)| < C_1 p(n) e^{-\eta n}, \mbox{ for some bounded constant } C_1 > 0.
$$
Thus it is sufficient to estimate the measure of this set. Denote the angle of stable direction of $A^E_k$ by $m_k(E)$. Since the stable direction is contracted at least by $1$, it is easy to see that
$$
|m_{k_1}(E) - s_{k_1}(E)| < C_2 e^{-\eta n} \mbox{ and } |m_{-k_2}(E) - u_{k_1}(E)| < C_2 e^{-\eta n}.
$$
Thus for large $n$, we can instead estimate the measure of the set
$$
\{ E : |m_{k_1}(E) - m_{-k_2}(E)| < C_3 p(n) e^{-\eta n} \}.
$$
We need to study the derivative of $m_k(E)$. Let $w_k(E) = \cot m_k(E)$.

Define the function $F^l(E,y) = A^{E-v_l} \cdot y$, where $A^{E-v_l}$ acts on $y \in \R \cup \{ \infty \}$ as a M\"obius transformation. We suppress $l$ from the notation $F^l$. Set $F_k(E,y) = A^E_k \cdot y$. Then $F_k(E,y) = F(E,F_{k-1}(E,y))$. Note we have $F_{k}(E,w_k(E)) = w_k(E)$, which allows us to calculate $\frac{dw_k}{dE}$. For simplicity, let $a_s(E) = F_s(E,w_k(E))$ and $a_0(E) = w_k(E)$. Note we also have $a_k(E)=a_0(E)=w_k(E)$. We first assume that $w_k(E) \neq \infty$. So we get
\begin{align*}
\frac{da_k}{dE}(E) & = \frac{\partial F}{\partial E}(E,a_{k-1}(E)) + \frac{\partial F}{\partial y} (E,a_{k-1}(E)) \frac{da_{k-1}}{dE}(E) \\
& = 1 + \frac{1}{a_{k-1}^2(E)} \frac{da_{k-1}}{dE}(E) \\
& = 1 + \sum^{k-1}_{l=1} \prod^{l}_{j=1} \frac{1}{a_{k-j}^2(E)} + \frac{dw_k}{dE}(E) \prod^{k}_{j=1} \frac{1}{a_{k-j}^2(E)} \\
& = \frac{dw_k}{dE}(E).
\end{align*}
Thus we get
$$
\frac{dw_k}{dE}(E) = \frac{1 + \sum^{k-1}_{l=1} \prod^{l}_{j=1} a_{k-j}^{-2}(E)}{1 - \prod^{k}_{j=1}a_{k-j}^{-2}(E)} = -\frac{1}{1 + [\frac{\pi}{2} - m_k(E)]^2} \frac{dm_k}{dE}(E).
$$
On the other hand, it is easy to see that
$$
A^E_k \binom{w_k(E)}{1} = \prod^{k}_{j=1} a_{k-j}(E) \binom{w_k(E)}{1},
$$
which implies that $\prod^{k}_{j=1} a_{k-j}(E)$ is the eigenvalue corresponds to the stable direction of $A^E_k$. Thus, we must have $|\prod^{k}_{j=1} a_{k-j}(E)| < 1$.
Hence,
$$
\frac{dm_k}{dE}(E) = \frac{1 + \sum^{k-1}_{l=1} \prod^{l}_{j=1} a_{k-j}^{-2}(E)}{\prod^{k}_{j=1} a_{k-j}^{-2}(E) - 1} \left( 1 + \Big[ \frac{\pi}{2} - m_k(E) \Big]^2 \right) > 0.
$$

Now to compute $\frac{dm_{-k}}{dE}(E)$, just note that $m_{-k}(E)$ is the unstable direction of $(A^E_{-k})^{-1}$ (assume again $w_{-k}(E) =\cot m_{-k}(E) \neq \infty$). Thus by $(A^E_{-k})^{-1} \cdot w_{-k}(E) = w_{-k}(E)$ and by exactly the same procedure for computing $\frac{dm_k}{dE}(E)$, we get
$$
\frac{dm_{-k}}{dE}(E) = \frac{1 + \sum^{k-1}_{l=1} \prod^{l}_{j=1} \hat a_{k-j}^{-2}(E)}{\prod^{k}_{j=1} \hat a_{k-j}^{-2}(E) - 1} \left( 1 + \Big[ \frac{\pi}{2} - m_{-k}(E) \Big]^2 \right),
$$
where $\hat a_s(E) = F_s(E,w_{-k}(E))$ and $\hat a_0(E) = w_{-k}(E)$. Now since $m_{-k}$ is the unstable direction, we must have $|\prod^{k}_{j=1} \hat a_{k-j}(E)| > 1$. Hence,
$$
\frac{dm_{-k}}{dE}(E) < 0.
$$
Note it might seem that if $w_k(E) = a_0(E) = 0$, then the formulas for $\frac{dm_ k}{dE}(E)$ may have some problems. But then we have $a_1(E)a_0(E) = (E - v_0 - \frac{1}{a_0(E)}) a_0(E) = -1$. Thus $\prod^{k}_{j=1} a_{k-j}(E)$ as eigenvalue of $A^E_k(0)$ always makes sense. Similarly for the case $w_{-k}(E) = 0$.

Now we estimate $|\frac{dm_{-k_1}}{dE}(E) - \frac{dm_{-k_2}}{dE}(E)|$. This is immediate since $0 < |\prod^{k_2}_{j=1} \hat a_{k_2-j}^{-2}(E) - 1| < 1$ implies that
$$
|\frac{dm_{-k_1}}{dE}(E) - \frac{dm_{-k_2}}{dE}(E)| > |\frac{dm_{-k_2}}{dE}(E)| > 1.
$$
This implies that whenever $m_{k_1}(E)$ is sufficiently close to $m_{-k_2}(E)$, $|m_{k_1}(E) - m_{-k_2}(E)| < C_3 p(n) e^{-\eta n}$ only for $E$ in an interval of size $C_3 p(n) e^{-\eta n}$.

If $w_k(E) = \infty$ or $w_{-k}(E) = \infty$, say $w_k(E) = \infty$, we define $b_s(E) = \frac{1}{F_s(E,w_k(E))}$ and $b_0(E) = \tan m_k(E)$. Then we have
\begin{align*}
\frac{db_k}{dE}(E) & = b_k(E)^2 [\frac{db_{k-1}}{dE}(E) - 1] \\
& = \frac{db_0}{dE}(E) \prod^{k-1}_{j=0} b_{k-j}^2(E) - \sum^{k-1}_{l=0} \prod^{l}_{j=0} b_{k-j}^2(E) \\
& = \frac{db_0}{dE}(E).
\end{align*}
Then we get
$$
\frac{dm_k}{dE}(E) = \frac{\sum^{k-1}_{l=0} \prod^{l}_{j=0} b_{k-j}^2(E)}{\prod^{k-1}_{j=0} b_{k-j}(E)^2 - 1}(1 + m_k(E)^2) > 0.
$$
Again as in the case $w_k(E) \neq \infty$, $\prod^{k-1}_{j=0} b_{k-j}(E)^2 = \prod^{k}_{j=1} b_{k-j}(E)^2 > 1$. From this we can again deduce that $|\frac{dm_k}{dE}(E)| > 1$. Then every other estimate follows the same way as in the case $w_{k_1}(E)$, $w_{-k_2}(E) \neq \infty$.

Finally, we claim that $m_{k_1}(E)-m_{-k_2}(E)=0$ has at most $2n$ roots in $\R\PP^1=\R/(\pi\Z)$ for $E\in[-C-2,C+2]$ and for all $k_1,k_2\in[1,n]$. This is due to the fact that in each spectral gap of the periodic operator $H_v$ of period $n$, the stable direction, $m_n(E)$, and unstable direction, $m_{-n}(E)$, of the $n$-step transfer matrix, $A^{E-v}_n$, can never meet. Since $\frac{dm_{n}}{dE}$ and $\frac{dm_{-n}}{dE}$ have different signs, it is necessary that both of them may wind around $\R \PP^1$ at most one time in each spectral gap. Hence, both of them may wind around $\R \PP^1$ at most $n$ times for $E\in[-C-2,C+2]$ since there are at most $n$ spectral gaps of $H_v$. Now since $k_1,\ k_2\in [1,n]$, both $m_{k_1}(E)$ and $m_{-k_2}(E)$ may wind around $\R\PP^1$ at most $n$ times. Thus, the claim follows from the fact that they move in different directions.

It's clear that the claim implies
$$
\{ E \in \CU \CH : |m_{k_1}(E) - m_{-k_2}(E)| < C_3 p(n) e^{-\eta n}, \mbox{ for some } k_1,\ k_2 \in[1,n] \}
$$
is of Lebesgue measure at most
$$
C_3 n^3 p(n) e^{-\eta n}.
$$
Together with the estimate for elliptic matrices, we get for large $n$,
$$
\mathrm{Leb}(\CE) \leq C n^3 p(n) e^{-\eta n},
$$
concluding the proof.
\end{proof}

We need the following two definitions to state and prove Theorem~\ref{t.4} below.

\begin{defi}\label{d.1}
We call an ergodic subshift $(\Omega,T,\mu)$ almost surely polynomially transitive if for every $\varepsilon > 0$, there exist $\delta > 0$, $C > 0$, and a sequence $n_k \to \infty$, such that for every $k$, there is $\Omega_k \subseteq \Omega$ with $\mu(\Omega_k) > 1 - \varepsilon$ such that for every $\omega \in \Omega_k$, we have
$$
\mu\left( \bigcup_{m = 0}^{C n_k^C} \left[ (T^m \omega)_{[0,n_k-1]} \right] \right) > \delta.
$$
Here, $[\eta_{[0,\ell-1]} ]$ denotes the cylinder set $\{ \omega \in \Omega : \omega_j = \eta_j , \; 0 \le j \le \ell-1 \}$.
\end{defi}

Note that every subshift that is uniformly polynomially transitive (in the sense that there is a polynomial $P$ such that every word of length $P(n)$ that occurs in an element of $\Omega$ contains all words of length $n$ that occur in elements of $\Omega$) is almost surely polynomially transitive with respect to every ergodic measure.

\begin{defi}\label{d.2}
We say that an ergodic subshift $(\Omega,T,\mu)$ is almost surely of polynomial complexity if for every $\varepsilon > 0$, there exists $C > 0$ such that for each $n \in \Z_+$, there is a set $W_n$ of words of length $2n+1$ that has cardinality at most $C n^C$ such that
$$
\mu\left( \bigcup_{w \in W_n} [w]_{-n,\ldots,n} \right) > 1 - \varepsilon.
$$
Here, $[w]_{-n,\ldots,n}$ denotes the cylinder set $\{ \omega \in \Omega : \omega_j = w_j , \; -n \le j \le n \}$.
\end{defi}

Note that every subshift of polynomial complexity (in the sense that there is a polynomial $P'$ such that the number of words of length $n$ occurring in elements of $\Omega$ is bounded by $P'(n)$) is almost surely of polynomial complexity with respect to every ergodic measure.

Then we have the following theorem.

\begin{theorem}\label{t.4}
Suppose the ergodic subshift $(\Omega,T,\mu)$ is aperiodic, almost surely polynomially transitive, and almost surely of polynomial complexity. Then, the associated density of states measure $dk_\mu$ is purely singular.
\end{theorem}

\begin{proof}
By the subadditive ergodic theorem, we have that for every $E \in \R$ and for every $\kappa > 0$, there exists $n_0$ such that for $n \ge n_0$, we have
$$
\mu \left( \{ \omega \in \Omega : \|A^{E-f}_n(\omega)\| \ge e^{(L(E)-\kappa)n} \} \right) > 1 - \kappa.
$$

By aperiodicity and Proposition~\ref{p.kotani}, $L(E) > 0$ for Lebesgue almost every $E \in \Sigma_\mu$. Thus, for every $\beta > 0$ and every $\eta>0$, we can find $\alpha > 0$, a compact subset $\Lambda$ of $\Sigma_\mu$ and $n_\Lambda \in \Z_+$ such that
$$
\mathrm{Leb}(\Sigma_\mu \setminus \Lambda) < \beta
$$
and
$$
\mu \left( \{ \omega \in \Omega : \|A^{E-f}_n(\omega)\| \ge e^{\alpha n}\} \right) > 1 - \eta \quad \text{ for every } n \ge n_\Lambda , \; E \in \Lambda.
$$

For $\gamma > 0$ and $\omega \in \Omega$, consider the set
$$
\Sigma_{\omega,\gamma} = \{ E \in \R : \exists u \text{ s.t. } H_\omega u = E u, \; |u(0)|^2+|u(1)|^2 = 1, \; |u(n)| \le \gamma (1 + |n|) \}.
$$
Obviously, $\Sigma_{\omega,\gamma} \subseteq \sigma(H_\omega)$ for each $\omega\in\Omega$ and each $\gamma>0$. Also, $\bigcup_{\gamma \in \Z_+}\Sigma_{\omega,\gamma}$ supports all spectral measures of $H_\omega$ and hence $\bigcup_{\omega \in \Omega'} \bigcup_{\gamma \in \Z_+}\Sigma_{\omega,\gamma}$ supports the density of states measure $dk_\mu$ if $\mu(\Omega') = 1$. Our goal is to find a full measure set $\Omega'\subseteq\Omega$ such that this union has zero Lebesgue measure. By Proposition 1, throughout this proof, we may work in a full measure subset $\Omega' \subseteq\Omega$ such that for each $\omega \in \Omega'$ and each $\gamma > 0$,
$$
\Sigma_{\omega,\gamma} \subseteq \sigma(H_\omega) = \Sigma_\mu.
$$

For $N \in \Z_+$, we also consider the set
$$
\{ E \in \R : \exists u \text{ s.t. } H_\omega u = E u, \; |u(0)|^2+|u(1)|^2 = 1, \; |u(n)| \le \gamma (1 + |n|), \; |n| \le N \}.
$$
and denote it by $\Sigma_{\omega,\gamma,N} $. Then we have
$$
\Sigma_{\omega,\gamma,N} \supseteq \Sigma_{\omega,\gamma,N + 1} \quad \text{ and } \quad \Sigma_{\omega,\gamma} = \bigcap_{N \ge 1} \Sigma_{\omega,\gamma,N}.
$$

Now by almost sure polynomial transitivity, for any $\varepsilon > 0$, there exist $\delta > 0$, $C' > 0$, and a sequence $n_k \to \infty$ such that for each $k$, there is a $\Omega_k' \subset \Omega'$ with $\mu(\Omega_k') > 1 - \frac{\varepsilon}{3}$, such that for every $\omega \in \Omega_k'$, we have
$$
\mu \left( \bigcup_{m = 0}^{C' n_k^{C'}} \left[ (T^m \omega)_{[0,n_k-1]} \right] \right) > \delta.
$$
Let $N_k = C' n_k^{C'} + n_k$. Then by almost sure polynomial complexity, for the $\varepsilon$ and $\delta$ above, there exists a constant $C$ such that for each $k$, there exists a
$$
\Omega_k \subset \Omega_k' \cap T^{-N_k} (\Omega_k') \mbox{ with } \mu (\Omega_k) > 1 - \varepsilon,
$$
which consists of cylinders of length $2 N_k + 1$ and the number of $2 N_k + 1$ cylinders in $\Omega_k$ is bounded by
$$
C N_k^C = C (C' n_k^{C'} + n_k)^C,
$$
which is again polynomially large in $n_k$.

Now by the discussion in the beginning of the proof, we have that for the $\varepsilon$ and $\delta$ above, there exist $\Lambda \subset \Sigma_\mu$, $\alpha > 0$, and $N \in \Z_+$ such that
$$
\mathrm{Leb} ( \Sigma_\mu \setminus \Lambda ) < \frac{\varepsilon}{2}
$$
and for each $E \in \Sigma$, the set
$$
D_E = \{ \omega \in \Omega : \|A^E_n(\omega)\| \ge e^{2 n \alpha},\ \forall n > N \}
$$
satisfies
$$
\mu(D_E) > 1 - \delta.
$$
Thus for $n_k > N$, we have that for each $\omega \in \Omega_k$ and each $E \in \Lambda$, there exist $m_1,\ m_2 \in [0, C' n_k^{C'}]$ such that
$$
[T^{m_1}(\omega)]_{[0,n_k-1]} \bigcap D_E \neq \varnothing, \quad [T^{m_2-N_k}(\omega)]_{[-n_k,-1]} \bigcap D_E \neq \varnothing.
$$
Thus we obtain
\begin{center}
either $\|A_{m_1}^E{(\omega)}\| > e^{n_k\alpha}$ or $\|A_{m_1+n_k}^E{(\omega)}\| > e^{n_k\alpha}$.
\end{center}
Similar, we also have
\begin{center}
either $\|A_{-(N_k-n_k-m_2)}^E{(\omega)}\| > e^{n_k\alpha}$ or $\|A_{-(N_k-m_2)}^E{(\omega)}\| > e^{n_k\alpha}$.
\end{center}
Now by Lemma~\ref{l.9}, we have for each $\omega \in \Omega_k$,
$$
\mathrm{Leb} \left( \Sigma_{\omega,\gamma,N_k} \cap \Lambda \right) < p(n_k) e^{-\alpha n_k}
$$
for some polynomial $p\in\R[X]$.

Notice that the set $\Sigma_{\omega,\gamma,N_k}$ only depends on $\omega_{-N_k}, \ldots, \omega_{N_k}$. By our choice of $\Omega_k$, there are only $C N_k^C$ many $2 N_k + 1$ cylinders. Thus we obtain for sufficiently large $n_k$,
$$
\mathrm{Leb} \left( \bigcup_{\omega \in \Omega_k} (\Sigma_{\omega,\gamma,N_k} \cap \Lambda) \right) < p'(n_k) e^{-\alpha n_k} < \frac{\varepsilon}{2},
$$
where $p'\in\R[X]$ is again a polynomial. Hence, we have for large $k$,
$$
\mu (\Omega_k) > 1 - \varepsilon
$$
and
$$
\mathrm{Leb} \left( \bigcup_{\omega \in \Omega_k} \Sigma_{\omega,\gamma} \right) < \mathrm{Leb} \left( \bigcup_{\omega \in \Omega_k} (\Sigma_{\omega,\gamma,N_k} \cap \Lambda) \right) + \mathrm{Leb} (\Sigma_\mu \setminus \Lambda) < \varepsilon.
$$
Now if we apply the result above to $\varepsilon = 2^{-\ell},\ \ell\in\Z_+$, then for each $\ell$, there is $\Omega^{(\ell)} \subset \Omega'$ such that
$$
\mu (\Omega^{(\ell)}) > 1 - 2^{-l} \mbox{ and } \mathrm{Leb} \left( \bigcup_{\omega \in \Omega^{(\ell)}} \Sigma_{\omega,\gamma} \right) < 2^{-l}.
$$
Thus if we set
$$
\Omega_\gamma = \limsup_{\ell \to \infty} \Omega^{(\ell)} = \bigcap_{s \ge 1} \bigcup_{\ell \ge s} \Omega^{(\ell)},
$$
then it is easy to see that
$$
\mu (\Omega_\gamma) = 1 \mbox{ and }  \mathrm{Leb} \left( \bigcup_{\omega \in \Omega_\gamma} \Sigma_{\omega,\gamma} \right) = 0.
$$
Now if we set
$$
\Omega'' = \bigcap_{\gamma \in \Z_+} \Omega_\gamma,
$$
then we have
$$
\mu (\Omega'') = 1 \mbox{ and } \mathrm{Leb} \left( \bigcup_{\gamma \in \Z_+} \bigcup_{\omega \in \Omega''} \Sigma_{\omega,\gamma} \right) = 0,
$$
which completes the proof.
\end{proof}

\subsection{Subshifts Generated by Shifts on Tori and Rectangular Grids}\label{subsec.3.5}

Consider a minimal translation $S_\alpha : \T^d \to \T^d$, $S_\alpha x = x + \alpha$. Consider a partition of $\T^d$ by finitely many rectangles with sides parallel to the standard hyperplanes,
$$
\T^d = \bigsqcup_{j = 1}^J R_j.
$$
More explicitly, we assume that each $R_j$ is of the form
$$
R_j = \{ x \in \T^d : \alpha_i^{(j)} \le x_i < \beta_i^{(j)} \}.
$$
Choose real numbers $\lambda_1, \ldots , \lambda_J$, not all equal, and write $A = \{ \lambda_1 , \ldots, \lambda_J \}$. Consider the sequence $s=s(0) \in A^\Z$ defined by
$$
s_n = \sum_{j = 1}^J \lambda_j \chi_{R_j}(S_\alpha^n 0),
$$
that is, $s_n = \lambda_j$ if and only if $S_\alpha^n 0 \in R_j$. Let $\Omega \subseteq A^\Z$ be the subshift generated by $s$, that is, take shifts and accumulation points. Since $S_\alpha$ is minimal, it follows that $\Omega$ is minimal as well with respect to the shift $T$. Moreover, $\Omega$ is also uniquely ergodic and the unique invariant measure is the push-forward of Lebesgue measure on $\T^d$ under
$$
x \mapsto  (s_n(x))_{n\in\Z},\ s_n(x)=\sum_{j = 1}^J \lambda_j \chi_{R_j}(S_\alpha^n x).
$$

\begin{lemma}\label{l.10}
The complexity function of $\Omega$ satisfies $p(n) \le Cn^d$. In particular, $\Omega$ has polynomial factor complexity, and hence it is almost surely of polynomial complexity with respect to the unique ergodic measure.
\end{lemma}

\begin{proof}
Since $S_\alpha$ is minimal, $1,\alpha_1,\ldots,\alpha_d$ are independent over the rationals. In particular, each $\alpha_j$ is irrational. Since $S_\alpha$ is the direct product of $d$ irrational rotations of the circle, the rectangles project to half-open intervals on the circles, and the complexity of any coding of an irrational rotation with respect to a partition by finitely many intervals is linearly bounded, the product must be bounded by the product of the bounds.
\end{proof}

We say that $\alpha$ is Diophantine if there are $C,\tau > 0$ such that
$$
\|\langle k,\alpha\rangle \|_{\R/\Z} \ge C \|k\|^{-\tau} \quad \text{ for all } k \in \Z^d \setminus \{ 0 \},
$$
where $\langle \cdot,\cdot \rangle $ denotes the usual scalar product and $\|\cdot\|_{\R/\Z}$ denotes the distance to the nearest integer. Let $\Delta^C_{\tau}$ denote the set of such $\alpha$. Then it is a standard result that for any $\tau > d-1$,
$$
\bigcup_{C > 0} \Delta^C_{\tau} \subset \R^d
$$
is of full Lebesgue measure.

\begin{lemma}\label{l.shiftdistr}
Suppose $\alpha$ is Diophantine. Then, $\Omega$ is almost surely polynomially transitive with respect to the unique ergodic measure.
\end{lemma}

\begin{proof}
Fix some $\beta \in (0,1)$. Given any $n \in \Z_+$, look at the subsets of $\T^d$ corresponding to $n$-cylinder subsets of $\Omega$. They are of the form
$$
[\lambda_{i_0}\cdots\lambda_{i_{n-1}}]_{0,\ldots,n-1}=R_{i_0}\cap S^{-1}_\alpha(R_{i_1})\cap\cdots\cap S^{-(n-1)}_\alpha(R_{i_{n-1}}).
$$
Thus they are intersections of pre-images of the rectangles under $S_\alpha$ and hence are rectangles themselves. The Lebesgue measure of such a rectangle is equal to the frequency of the corresponding word of length $n$ defining the cylinder set. Since there are at most $Cn^d$ many such words, the measure of the cylinder sets that each have measure less than $\beta / (Cn^d)$ is at most $\beta$. Thus, look at all cylinder sets that each have measure at least $\beta / (Cn^d)$. Their union has measure at least $1 - \beta$.

Consider such a word whose frequency is at least $\beta /(Cn^d)$. The corresponding rectangle has the same Lebesgue measure and (since the sides are bounded above by $1$) its in-radius is at least $\frac12 \beta/(C n^d)$. Thus each of these rectangles contains a ball with measure at least $c n^{-d^2}$, which is polynomially small in $n$.

Now we claim that any Diophantine translation orbit $S^k_\alpha(x)$ becomes $\gamma$-dense in time polynomial in $\gamma^{-1}$. In other words, for each $x\in\T^d$, for each ball $\CB$ in $\T^d$ with measure at least $\gamma$, there exists a $k$ which is polynomially large in $\gamma^{-1}$ such that $S^k_\alpha(x)\in \CB$. Clearly, by compactness, this polynomial can be chosen to be independent of $x\in\T^d$. Note also once the orbit, $S^k_\alpha(x)$, visits the corresponding rectangle, then the $n$-cylinder set
$$
[s_k(x)\cdots s_{k+n-1}(x)]_{0,\ldots,n-1}\subset\Omega
$$
represents the whole rectangle. Assume this claim and apply it with $\gamma = cn^{-d^2}$, then any orbit visits each rectangle above in time polynomially large in $n$. And this polynomial is independent of the initial point. Now if we go back to the subshift setting, then the result follows with $\delta=1-\beta$ and the $(\varepsilon,n)$-independent $\Omega_n = \Omega$.

For the proof of the claim, we need to solve the following cohomological equation:
$$
h(x+\alpha)-h(x)=f(x)-\int_{\T^d}fdx
$$
for some $f\in C^k(\T^d)$. By Fourier expansion, it is easy to see that for any $\alpha \in \Delta_\tau$, if $k$ is sufficiently large, there is a smooth solution of the above equation, say $h \in C^l(\T^d)$ for some $l$ less than $k$. Furthermore, we have the estimate
$$
\|h\|_{C^l}\lesssim \|f\|_{C^k}.
$$
From this, we can readily get the following estimate
$$
\|S_N(f) - N \int f\|_\infty \lesssim \|h\|_{C^l} \lesssim \|f\|_{C^k},
$$
where
$$
S_N(f) = \sum^{N-1}_{k=0} f(S^k_\alpha(x)) = \sum^{N-1}_{k=0} f(x + k\alpha).
$$
Now for any ball $\CB$ with radius $\gamma$, we can approximate the characteristic function $\chi_\CB$ by a non-negative function $f\in C^k(\T^d)$ with $\|f\|_{C^k} \lesssim \gamma^{-C'}$, which is also supported in $\CB$. Note also $\int_{\T^d} fdx\approx \gamma^d$. Thus, we have
$$
S_N(f)\gtrsim N\gamma^d-\gamma^{-C'}>0
$$
for some $N>\gamma^{-(C'+d)}$, which is polynomially large in $\gamma^{-1}$. This completes the proof of Lemma~\ref{l.shiftdistr}.
\end{proof}

\begin{corollary}\label{c.qpsingdos}
If $\alpha$ is Diophantine and the subshift $\Omega$ is generated as above, the density of states measure associated with the unique ergodic measure on $\Omega$ is singular.
\end{corollary}

\begin{proof}
Since by construction the subshift is aperiodic, the corollary follows from Theorem~\ref{t.4} and Lemmas~\ref{l.10} and \ref{l.shiftdistr}.
\end{proof}

\subsection{Subshifts Generated by the Skew-Shift and Rectangular Grids}

Consider the standard skew-shift $\tilde S_\alpha : \T^2 \to \T^2$, $\tilde S_\alpha(x,y) = (x + \alpha,x + y)$ with $\alpha$ irrational. This map is strictly ergodic with Lebesgue measure as the unique invariant measure. As before, consider a partition of $\T^2$ by finitely many rectangles of the form
$$
R_j = \{ x \in \T^2 : \alpha_i^{(j)} \le x_i < \beta_i^{(j)} \}.
$$
Choose real numbers $\lambda_1, \ldots , \lambda_J$, not all equal, and write $A = \{ \lambda_1 , \ldots, \lambda_J \}$. Consider the sequence $s \in A^\Z$ defined by
$$
\tilde s_n = \sum_{j = 1}^J \lambda_j \chi_{R_j}(\tilde S_\alpha^n 0),
$$
that is, $\tilde s_n = \lambda_j$ if and only if $\tilde S_\alpha^n 0 \in R_j$. Let $\Omega \subseteq A^\Z$ be the subshift generated by $\tilde s$, that is, take shifts and accumulation points. Since $\tilde S_\alpha$ is strictly ergodic, it follows that $\Omega$ is strictly ergodic as well with respect to the shift $T$. Again, the unique invariant measure is the push-forward of Lebesgue measure on $\T^2$ under
$$
(x,y) \mapsto  \left(\sum_{j = 1}^J \lambda_j \chi_{R_j}(\tilde S_\alpha^n (x,y))\right)_{n\in\Z}.
$$

\begin{lemma}\label{l.12}
The complexity function of $\Omega$ satisfies $p(n) \le Cn^3$. In particular, $\Omega$ has polynomial factor complexity, and hence it is almost surely of polynomial complexity with respect to the unique ergodic measure.
\end{lemma}

\begin{proof}
Note that
$$
\tilde S_\alpha^n (x,y) = \left(x + n \alpha, nx + \frac{n(n-1)}{2} \alpha + y\right) = \left(x + n \alpha, n^2 \frac{\alpha}{2} + n \left(  x - \frac{\alpha}{2} \right) + y \right).
$$
By minimality, we can compute the factor complexity for any element of $\Omega$ since all elements have the same set of finite factors. Let us choose the element that equals the coding of the point $(x,y) = (\frac{\alpha}{2},0)$. Again, we estimate the factor complexity by the product of the complexities of the projections to the coordinates with respect to the projections of the rectangles. In the first component ($\frac{\alpha}{2} + n \alpha$) we get a linear complexity as we are coding an irrational rotation with respect to a partition of the circle into finitely many half-open intervals. In the second component ($n^2 \frac{\alpha}{2}$), we have at most quadratic complexity. This follows since we take a subsequence of the coding of ($m \frac{\alpha}{2}$) which generates at most quadratically many words and then drop symbols, which won't increase the number of different words of a given length.
\end{proof}

\begin{lemma}\label{l.13}
For every Diophantine $\alpha\in\R$, $\Omega$ is almost surely polynomially transitive with respect to the unique ergodic measure.
\end{lemma}

\begin{proof}
This is a modification of the proof of Lemma~\ref{l.shiftdistr}. Recall that two ingredients were used. First, the sets on $\T^2$ corresponding to $n$-cylinder subsets of the subshift contain balls of at worst polynomially (in $n$) small radius and the collection of these sets on $\T^2$ has measure at least $\delta > 0$, which must be $n$-independent. Second, the orbit of the torus transformation becomes $\gamma$-dense in time polynomial in $\gamma^{-1}$.

For the first ingredient, note that by Lemma~\ref{l.12} and by exactly the same argument as in the proof of Lemma~\ref{l.shiftdistr}, for each $\beta\in(0,1)$, the union of all $n$-cylinder sets that each have measure at least $\beta / (Cn^3)$ has measure at least $1 - \beta$. Note in this case, the transformation does not preserve the rectangular structure. But it does send a rectangle to a parallelogram. Thus the sets on $\T^2$ corresponding to $n$-cylinder subsets of $\Omega$, which are of the form
$$
[\lambda_{i_0}\cdots\lambda_{i_{n-1}}]_{0,\ldots,n-1}=R_{i_0}\cap \tilde S^{-1}_\alpha(R_{i_1})\cap\cdots\cap \tilde S^{-(n-1)}_\alpha(R_{i_{n-1}}),
$$
are convex polygons contained in the original rectangle $R_{i_0}$. Thus, the sets on $\T^2$ corresponding to $n$-cylinder sets with measure at least $\beta/(Cn^3)$ in $\Omega$ must contain balls with radius $cn^{-3}$. And the union of such sets has measure at least $1-\beta$.

For the second ingredient, we want to estimate the following term
$$
\left\|S_Nf-N\int f\right\|_\infty
$$
for $f\in C^k(\T^2)$ with some large $k\in\Z^+$. To get the desired estimate, we need to decompose $f$ as
$$
f=\left(f-\int_\T f dy\right)+\int_\T fdy:=f_1+f_2
$$
and estimate $f_1$ and $f_2$ separately. For the estimate of $f_2$, it depends only on $x$. Thus the claim in Lemma~\ref{l.shiftdistr} gives us that
$$
\left\|S_N f_2 - N \int f\right\|_\infty\lesssim \|f\|_{C^k}
$$
For the estimate of $f_1$, we need to use the following estimate from \cite[Theorem~11]{AFU}

$$
\left\|S_{N_l} f_1\right\|_\infty \lesssim N_l^{\frac12} \|f\|_s\lesssim N_l^{\frac12}\|f\|_{C^k}
$$
along some $\alpha$-dependent sequence $N_l\rightarrow\infty$, where $\|\cdot\|_s$ is the Sobolev norm with index $s<k$. See \cite{AFU, FF} for more detailed information. The last inequality follows from a standard relation between the Sobolev norm and the $C^k$ norm.

Combining the estimates for $f_1$ and $f_2$, we get
\begin{equation}\label{e.flamforn}
\left\| S_{N_l}f - N_l \int f \right\| \lesssim N_l^{\frac12} \|f\|_{C^k}.
\end{equation}

Now for a fixed ball $\CB$ in $\T^2$ with radius $\gamma$, we can consider a smooth function $f$ that is non-negative and close to the characteristic function $\chi_\CB$ and is supported in $\CB$. Then, the integral of $f$ will be of order $\gamma^2$. The norm on the right-hand side of \eqref{e.flamforn} will be of order $\gamma^{-C'}$. Now the estimate \eqref{e.flamforn} implies that if
$$
N_l \gamma^2 \gtrsim N_l^{\frac12} \gamma^{-C'},
$$
then for every $(x,y) \in \T^2$, at least one of the first $N_l$ iterates of $(x,y)$ must meet the support of $f$. Simplifying, we find that
$$
N_l = C \gamma^{-2(C'+2)}
$$
is sufficient for a suitable constant $C$. Applying this to $\gamma=cn^{-3}$, the result follows with the sequence $n_l=[cN_l^{\frac{1}{6(C'+2)}}]$, $\delta=1-\beta$ and $(\varepsilon,n_l)$-independent $\Omega_l=\Omega$.
\end{proof}

\subsection{Interval Exchange Transformations}

Denote the unit interval $[0,1)$ by $I$. Given an irreducible permutation $\pi \in \mathcal{S}_r$ and lengths $\lambda_1,\ldots,\lambda_r > 0$ with $\sum_{j=1}^r \lambda_j = 1$, we consider the associated interval exchange transformation $T_{\pi,\lambda} : I \to I$. We either look at the standard coding of the trajectories by assigning $r$ real numbers to the $r$ intervals, or we consider a piecewise constant sampling function $f : I \to \R$ and code according to $f \circ T_{\pi,\lambda}^n$. Of course, the two viewpoints are equivalent. In any event, we obtain two-sided symbolic sequences that generate subshifts as before. These subshifts will be minimal if the IDOC (infinite distinct orbit condition) holds, and they will Lebesgue almost surely be uniquely ergodic and (if $\pi$ is not a rotation) weakly mixing. The push-forward of Lebesgue measure under $x \mapsto \{f(T_{\pi,\lambda}^nx)\}$ is always invariant and hence typically it is the unique ergodic measure. We will assume throughout that $\Omega$ is aperiodic, otherwise complexity and transitivity issues are trivial.

\begin{lemma}\label{l.14}
For the natural coding, we have that the complexity function of $\Omega$ satisfies $p(n) = (r-1)n + 1$. In particular, for any coding by a piecewise continuous sampling function, $\Omega$ has polynomial factor complexity, and hence it is almost surely of polynomial complexity with respect to every ergodic measure.
\end{lemma}

\begin{proof}
This is well known.
\end{proof}

\begin{lemma}\label{l.15}
$\Omega$ is almost surely polynomially transitive with respect to the push-forward of Lebesgue measure {\rm (}which is almost surely the unique ergodic measure{\rm )}.
\end{lemma}

\begin{proof}
Without loss of generality we consider the case of natural coding. We will show that for every $\varepsilon > 0$, there exist $\delta>0$ and $c>0$ such that for every $n$, there is $\Omega_n \subseteq \Omega$ with $\mu(\Omega_n) > 1 - \varepsilon$ such that for every $\omega \in \Omega_n$, we have
$$
\mu\left( \bigcup_{m = 0}^{c n} \left[ (T_{\pi,\lambda}^m \omega)_{[0,n-1]} \right] \right) > \delta.
$$

Note that $n$-cylinder sets are obtained in the geometric picture (on $I$) by intersecting intervals as follows:
$$
[\eta_0 \ldots \eta_{n-1}]_{0,\ldots,n-1} = I_{\eta_0} \cap T_{\pi,\lambda}^{-1}(I_{\eta_1}) \cap \cdots \cap T_{\pi,\lambda}^{-(n-1)}(I_{\eta_{n-1}}).
$$
Thus, when considering the $n$-cylinder sets visited by a piece of a $T_{\pi,\lambda}$-orbit, we can consider pieces of Rokhlin towers obtained by starting at some level corresponding to an $n$-cylinder set, denoting the length of this interval by $\ell$, and take pre-images under $T_{\pi,\lambda}$. Since $\Omega$ is aperiodic, we will eventually encounter a pre-image that contains a point of discontinuity of $T_{\pi,\lambda}^{-1}$, and then we stop the iteration short of that. The height of this piece of the tower will be denoted by $h$. It is known \cite{B85} that there is a constant $C$ such that for every $n$, the number of values the length $\ell$ can take for $n$-cylinder sets is bounded by $C$. Note also that $T_{\pi,\lambda}^{-1}$ has at most $r-1$ discontinuity points. Finally, we use that the number of $n$-cylinders is $(r-1)n + 1$.

Now let us partition the $n$-cylinder sets into a group of good ones and two groups of bad ones. The first group of bad ones consists of those that have length $\le \frac{\varepsilon}{2rn}$ (they make up measure less than $\frac{\varepsilon}{2}$ by the complexity result). The second group of bad ones are those for which $h \ell < \frac{\varepsilon}{2(r-1) C}$ (so that $h < \frac{\varepsilon}{2(r-1) C \ell}$). Their total measure is bounded as follows:
$$
\mathrm{Leb} \left( \bigcup_{\text{disc.\ of } T_{\pi,\lambda}^{-1}} \; \bigcup_{\text{values of } \ell} \; \bigcup_{k=0}^h \text{level}(k) \right) \le (r-1) C \frac{\varepsilon}{2(r-1) C \ell} \ell = \frac{\varepsilon}{2}.
$$
All remaining $n$-cylinders are good and their union will form the set $\Omega_n$. Each of them has length $\ell$ at least $\frac{\varepsilon}{2rn}$ and their union has measure at least $1 - \varepsilon$. The tower starting from each of them and going down has height at least $\frac{\varepsilon}{2(r-1) C \ell}$, consists of disjoint intervals of length at least $\ell$. Thus, taking the partial piece of the tower of height
$$
\frac{\varepsilon}{2(r-1) C \ell} \le \frac{\varepsilon}{2(r-1) C} \frac{2rn}{\varepsilon} = \frac{r}{(r-1)C} n,
$$
which is polynomially bounded in $n$ as desired, we obtain a set of measure at least $\delta := \frac{\varepsilon}{2(r-1) C} > 0$.
\end{proof}

\section{One-Frequency Quasi-Periodic Potentials}\label{sec.4}

In this section we consider one-frequency quasi-periodic potentials. They arise from the general framework by setting $\Omega = \T = \R / \Z$ and $T : \T \to \T$, $\omega \mapsto \omega + \alpha$ with some $\alpha \in \T$. Note that $T$ is minimal if and only if $\alpha$ is irrational.

Given a bounded measurable sampling function $f : \T \to \R$, we consider Schr\"odinger operators $\{ H_\omega \}_{\omega \in \T}$ defined as before; compare~\eqref{e.oper}. Deviating slightly from the notation introduced in Section~\ref{sec.2}, we will denote
$$
\Sigma = \bigcup_{\omega \in \T} \sigma(H_\omega).
$$
This notation coincides with the one in Section~\ref{sec.2} when $\alpha$ is irrational (and $f$ is continuous). When $\alpha$ is rational, the set $\Sigma$ (which was not defined in Section~\ref{sec.2} because $T$ is not minimal in this case) as defined here will be convenient when we approximate irrational $\alpha$'s by rational ones. Similarly, the density of states measure $dk$ on $\R$ is given by
$$
\int g \, dk = \int_\T \langle \delta_0 , g(H_\omega) \delta_0 \rangle \, d\omega,
$$
which coincides with the previous definition when $\alpha$ is irrational, but is needed also for rational $\alpha$ in this form for our approximation purposes.

\begin{lemma}\label{l.1}
Let $f \in C(\T,\R)$. Then,
$$
\Sigma = \R \setminus \{ E : (T, A^{E-f}) \text{ is uniformly hyperbolic} \}.
$$
\end{lemma}

\begin{proof}
This follows from Johnson's theorem \cite{J} (see Proposition~\ref{p.johnson}) when $\alpha \not\in \Q$, and it is trivial when $\alpha \in \Q$.
\end{proof}

\begin{lemma}\label{l.2}
The map $C(\T,\R) \times \T \ni (f,\alpha) \mapsto \Sigma$ is continuous with respect to the Caratheodory metric, $d$, on the compact subsets of $\R$.
\end{lemma}

\begin{proof}
Assume $\|f\|_\infty<c$. We need to use the following description of spectrum: $E \in \Sigma$ if and only if there exists $\omega \in \T$ such that for every $\eta > 0$, there are $m = m(c,\eta) \in \Z_+$, a unit vector $u = (u_n)_{n \in \Z} \in \ell^2(\Z)$, and $M \in \Z$ with $u_n = 0$ for $|n - M| > m$, such that $\|(H_\omega - E)u\| \leq \eta$.

Now to prove the lemma, it suffices to show the following. Given $\lim_{s \to \infty} (f_s,\alpha_s) = (f,\alpha)$ in $C(\T,\R) \times \T$, let $\Sigma_s$ be the spectrum corresponding to $(f_s, \alpha_s)$. Let $E_s \in \Sigma_s$ be such that $\lim_{s \to \infty} E_s = E$, then $E \in \Sigma$.

By definition, for every $s$, there exists $\omega_s \in \T$ such that $E_s \in \sigma(H_{f_s, \alpha_s, \omega_s})$. Then, after a translation in $\omega_s$, say $\omega_s + M_{l,s} \alpha_s$, we get that for each $l \in \Z_+$, there are $m_l \in \Z_+$ and unit vectors $u^{l,s}$ with $u^{l,s}_n = 0$ for $|n| > m_l$ such that
$$
\|(H_{\alpha_s, f_s, \omega_s + M_{l,s} \alpha_s} - E_s) u_{l,s}\| < \frac{1}{2^{l+1}}.
$$
By passing to a subsequence we may assume that $\lim_{s \to \infty} \omega_s + M_{l,s} \alpha_s = \omega^l$. Thus for some large $s_l$, we have
$$
\|(H_{\alpha, f, \omega^l} - E) u_{l,s_l}\| < \frac{1}{2^{l}}.
$$
Again we may assume that $\lim_{l \to \infty} \omega^l = \omega$. Now for any $\eta > 0$, we can choose a sufficiently large $l$ such that $\|f(\omega + \cdot) - f(\omega^l + \cdot)\|_{\infty} + \frac{1}{2^l} < \eta$. Thus we have
$$
\|(H_{\alpha, f, \omega} - E) u_{l,s_l}\| < \eta,
$$
which implies that $E \in \sigma(H_{\alpha, f, \omega}) \subset \Sigma$.
\end{proof}

\begin{lemma}\label{l.3}
For every $r > 0$, the map $\T \times (B_r(L^\infty(\T,\R)), \|\cdot\|_1) \ni (\alpha, f) \mapsto k$ is continuous with respect to uniform convergence. Here, $B_r(L^\infty(\T,\R))$ is the open ball around the origin with radius $r$ in $L^\infty(\T,\R)$.
\end{lemma}

\begin{proof}
To prove this lemma, it suffices to show that for any uniformly bounded convergent sequence $(\alpha_s, f_s) \in \T \times L^1(\T,\R)$, thus $\lim_{s \to \infty} (\alpha_s, f_s) = (\alpha, f)$ for some $(\alpha, f)$, we can find a subsequence $(\alpha_{s_l}, f_{s_l})$ such that the corresponding IDS $k_{s_l}$ satisfies
$$
\lim_{l \to \infty} \|k_{s_l} - k\|_{\infty} = 0.
$$

Since all sampling functions $g$ in this proof satisfy $\|g\|_\infty < r$, all IDS's are distribution functions of some probability measures supported in $[-r-2, r+2]$. Thus it is sufficient to show that for a countable dense subset $\CE$ in $[-r-2, r+2]$, $\lim_{l \to \infty} k_{s_l}(E) = k(E)$ for each $E \in \CE$. In particular, it will be sufficient to show that $k_{s_l}(E)$ converges to $k$ pointwise in $[-r-2, r+2]$.

Now we may start with a sequence $(\alpha_s, f_s)$ such that $f_s\rightarrow f$ in $L^1$ and pointwise.

Let $\HH$ be the upper half plane in $\C$. It is well-known that for each $E \in \HH$, the cocycle $(\alpha, A^{E-f})$ is \emph{uniformly hyperbolic}, the unstable direction $u(E, \alpha, f, \omega)$ is in $\HH$ and the IDS $k$ can be written as
$$
1 - \frac{1}{\pi} \int_\T \arg u(E, \alpha, f, \omega) \, d\omega.
$$
Here, $\arg u(E, \alpha, f, \omega)$ is well defined since $u(E, \alpha, f, \omega)\in\HH$ for all $E \in \HH$. Then we have the following facts for $u(E, \alpha, f, \omega)$:

\begin{itemize}

\item For any bounded set $\CK \in \HH \times L^{\infty}(\T,\R)$, $\{ u(E, \alpha, f, \omega) : (E, f) \in \CK \}$ is a bounded subset in $\HH$. Here boundedness in $\HH$ is with respect to the hyperbolic metric.

\item For almost every $\omega \in \T$, as a function on $\HH$, $u(\cdot, \alpha_s, f_s, \omega) \to u(\cdot, \alpha, f, \omega)$ in the open compact topology as $s \to \infty$.

\end{itemize}

Now, mimicking the proof of \cite[Lemma 1]{AD}, replacing the Lyapunov exponents $L(E)$ by $k(E)$ and using the expression $k(E) = 1 - \frac{1}{\pi} \int_\T \arg u(E,f,\alpha,\omega) \, d\omega$ for $E \in \HH$, we can show that
$$
\lim_{s \to \infty} \int^{r+2}_{-r-2} |k(E; \alpha_s, f_s) - k(E; \alpha, f)| \, dE = 0.
$$
Thus there exists some subsequence $\{s_l\}_{l \in \Z^+}$ of $s$ such that $k(E; \alpha_{s_l}, f_{s_l}) \to k(E; \alpha, f)$ for almost every $E \in [-r-2, r+2]$ as $l \to \infty$. This completes the proof.
\end{proof}

\begin{lemma}\label{l.4}
Fix arbitrary $\alpha \in \T$ and $r > 0$. Then, for every $\delta > 0$ and $\varepsilon > 0$, the set
$$
\left\{ f \in L^\infty(\T,\R): dk_f(\{ E: L_f(E) < \delta \}) > \varepsilon \right\}
$$
is open in $(B_r(L^\infty(\T,\R)),\|\cdot\|_1)$.
\end{lemma}

\begin{proof}
Let $\delta > 0$, $\varepsilon > 0$ and $f$ be such that $dk_f \{ E: L_f(E) < \delta \} > \varepsilon$. Then by Lemma~\ref{l.3} and the equivalence between the uniform convergence of $k(E)$ and weak-$*$ convergence of $dk$, there exists $N_1 \in \Z^+$ such that for any $g \in B_\frac{1}{N_1}(f)$, we have
$$
dk_g (\{ E: L_f(E) < \delta \}) > \varepsilon,
$$
where $B_{a}(f)$ is the open ball around $f$ in $(B_r(L^\infty(\T,\R)),\|\cdot\|_1)$ with radius $a$. Note that here we use the upper semi-continuity of the Lyapunov exponent with respect to $E$ to conclude that $\{ E: L_f(E) < \delta \}$ is open for any $f \in L^\infty$ and $\delta > 0$.

On the other hand, it is not hard to see that for each $E \in \R$, the map
$$
(B_r(L^\infty(\T,\R)),\|\cdot\|_1) \to \R^+ \cup \{0\}, \; f \mapsto L_f(E),
$$
is also upper semi-continuous. Thus for any $E \in \{ E: L_f(E) < \delta \}$, there exists $n \in \Z^+$ such that if $\|g - f\|_1 < \frac{1}{n}$, then $L_g(E) < \delta$. Thus if we set
$$
D_n = \{ E: L_g(E) < \delta,\ \; \forall \, g \in B_{\frac{1}{n}}(f) \},
$$
then clearly $D_n \subset D_{n+1}$, $n \ge 1$ and
$$
\bigcup_{n \ge 1} D_n = \{ E : L_f(E) < \delta \}.
$$
Hence, there exists $N_2 \in \Z^+$ such that for every $n \ge N_2$ and every $g\in B_\frac1{N_1}(f)$
$$
dk_g (D_n) > \varepsilon.
$$
Set $N = \max\{N_1, N_2\}$. Then we have for any $g \in B_\frac{1}{N}(f)$,
$$
dk_g (\{ E: L_g(E) < \delta \}) > \varepsilon,
$$
concluding the proof.
\end{proof}

Now we are ready to show that
\begin{theorem}\label{t.1}
For every $\alpha \not\in \Q$, the set
$$
\left\{ f: dk_f(\{ E : L_f(E) = 0\}) = 1\right\}
$$
is a dense $G_\delta$ in $(C(\T,\R),\|\cdot\|_\infty)$.
\end{theorem}

\begin{proof}
Clearly, Lemma~\ref{l.4} implies that for any $n \in \Z^+$, the set
$$
D_n := \left\{ f \in C(\T,\R): dk_f \left( \left\{ E: L_f(E) < \tfrac{1}{n} \right\} \right) > 1 - \tfrac{1}{n} \right\}
$$
is open in $(C(\T,\R), \|\cdot\|_\infty)$. Thus to prove Theorem~\ref{t.1}, it suffices to show that
$$
\left\{ f \in C(\T,\R): dk_f(\{ E: L_f(E) = 0 \}) = 1 \right\}
$$
is dense in $(C(\T,\R), \|\cdot\|_\infty)$.

Fix arbitrary $f = f_1\in C(\T,\R)$ and $\delta = \delta_1 > 0$. We can pick a step function $s : \T \to \R$ which jumps only at rational numbers and obeys $\|s - f\|_\infty < \frac{\delta_1}{4}$. Since $\alpha$ is irrational, Kotani's Theorem and \cite[Theorem 10]{DL2} imply that $\Sigma_s = \{E : L_s(E) = 0\}$, which is of Lebesgue measure zero. Thus $dk_s$ is concentrated on a set of Lebesgue measure zero. By Lemma~\ref{l.4} and the proof of \cite[Lemma 3]{AD}, we can find $f_2 \in C(\T,\R)$ such that $\|f_2 - s\|_\infty < \frac{\delta_1}{4}$ and $f_2 \in D_2$. Hence
$$
f_2 \in B_{\frac{\delta_1}{2}}(f) \cap D_2,
$$
where $B_a(f)$ denotes the open ball around $f$ in $(C(\T,\R),\|\cdot\|_\infty)$ with radius $a$. Let $\overline B_a(f)$ be the closure of $B_a(f)$. Clearly we can find a $\delta_2 < \frac{\delta_1}{2}$ such that
$$
\overline B_{\delta_2}(f_2) \subset B_{\delta_1}(f_1) \cap D_2.
$$

Now by the same procedure as above and by induction, we can find a sequence of functions $\{f_k\}_{k \ge 1}$ and a sequence of positive numbers $\{\delta_k\}_{k \ge 1}$ such that
$$
\overline B_{\delta_{k+1}} (f_{k+1}) \subset B_{\delta_k} (f_k) \cap D_{k+1}
$$
with $\delta_k < \frac{\delta_{k-1}}{2}$.

Thus there is an $f_\infty \in C(\T,\R)$ such that
$$
\lim_{k \to \infty}f_k = f_\infty \in B_\delta(f).
$$
Since $\{f_k\}_{k \ge n}$ is contained in a ball whose closure is contained in $D_n$, we have
$$
f_\infty \in D_k, \quad \forall k \ge 1.
$$
Thus,
$$
dk_{f_\infty} (\{E : L_{f_\infty}(E) = 0\}) = 1.
$$
Now for each $n \in \Z^+$, $D_n$ is open and dense. Hence, we have

$$
\left\{ f: dk_f(\{ f \in C(\T,\R) : L_f(E) = 0\}) = 1 \right\} = \bigcap_{n \ge 1} D_n
$$
is a dense $G_\delta$ in $(C(\T,\R),\|\cdot\|_\infty)$.
\end{proof}

\begin{corollary}\label{c.1}
For every $\alpha \not\in \Q$, the set
$$
\{ f \in C(\T,\R) : dk \text{ is singular} \}
$$
is a dense $G_\delta$ subset of $(C(\T,\R),\|\cdot\|_\infty)$.
\end{corollary}

\begin{proof}
In our setting we may apply \cite[Theorem 1]{AD}, which says that there is a dense $G_\delta$ set in $(C(\T,\R),\|\cdot\|_\infty)$ such that the corresponding Lyapunov exponents are positive for almost every $E$ in the spectrum. Combined with Theorem~\ref{t.1}, this implies Corollary~\ref{c.1}.
\end{proof}

We remark that using the singularity result for minimal shifts on higher-dimensional tori with Diophantine shift vector and rectangular finite grids on the torus (Corollary~\ref{c.qpsingdos}, presented in Subsection~\ref{subsec.3.5}) in place of \cite{DL2}, we have the analogous singularity result for generic continuous functions for such Diophantine multi-frequency models.

\bigskip

Theorem~\ref{t.2} below shows that in our context, having zero measure spectrum is strictly stronger than having singular density of states measure. Let us start with the following lemma.

\begin{lemma}\label{l.5}
Given $\alpha \in \Q$ and $f \in C(\T,\R)$, fix a non-degenerate interval $I \subset \Sigma_{\alpha, f}$. Then, for every $\varepsilon > 0$, there exists $\delta > 0$ such that for every $\beta \in B_\delta(\alpha) \cap \Q$, $g \in B_\delta (f)$, and $r > 0$, we can modify $g$ on a collection of intervals, $J_i \subset \T$, $1 \le i \le k$, which satisfy $\beta + \bigcup J_i = \bigcup J_i$ and
$$
\sum^k_{i=1} |J_i| < r,
$$
to $h \in C(\T,\R)$ such that
$$
\|g - h\|_\infty \le \varepsilon \ \mathrm{and}\ I \subset \Sigma_{\beta, h}.
$$
\end{lemma}

\begin{proof}
Assume for simplicity that $\varepsilon$ is much smaller than $|I|$. By Lemma~\ref{l.2}, there exists $\delta > 0$ such that for all $\beta \in B_\delta(\alpha)$ and $g \in B_\delta(f)$, we have
$$
d(\Sigma_{\alpha, f}, \Sigma_{\beta, g}) < \frac{\varepsilon}{2}.
$$
Assume $\beta = \frac{p}{q}$ is rational. Let us list all the gaps of $\Sigma_{\beta, g}$ in $I$ as $G_i$, $1 \le i \le l$. Then each of them is of size smaller than $\varepsilon$. For simplicity, we assume that they lie entirely in $I$. Other cases can be treated similarly. Let $\T \ni \omega_i$, $1 \le i \le l$ be such that for each $i$, one band of $\Sigma_{\omega_i}$ lies in the left end of $G_i$. For simplicity, we assume all $\omega_i$, $1 \le i \le l$ are distinct. Again, other cases can be treated similarly. Then we can take small intervals $J_i + n \frac{p}{q} \subset \T$ centered around $\omega_i + n \frac{p}{q}$, $0 \le n \le q-1$, mutually disjoint in $\T$, so that for the given $r > 0$, we have
$$
\sum_{i=1}^{l} q |J_i| < r.
$$
Let
$$
\varphi(x) = \begin{cases} 1+x & {\rm if} \ -1 \le x \le 0, \\ 1-x & {\rm if}\ 0 < x \le 1, \\ 0 & {\rm otherwise}. \end{cases}
$$
Now if we define $h$ as
$$
h(\omega) = g(\omega) + \sum^{l}_{i=1} \sum^{q-1}_{n=0} \varepsilon \varphi \left[ \frac{1}{|J_i|} (\omega - \omega_i - n \frac{p}{q}) \right],
$$
then it is clear that $\|g-h\|_\infty \le \varepsilon$. On the other hand, for $\omega\in J_i$, let $D_{g,\omega}$ denote the band of $\Sigma_{g,\omega}$ lying to the left of the gap $G_i$ and such that $d(D_{g,\omega},G_i)$ is minimal among all the bands lying to the left of $G_i$.  Let $d_i$ denote the right boundary of $D_{g,\omega_i}$. Thus $d_i$ is also the left boundary of $G_i$. We assume $J_i$ is so small that $\bigcup_{\omega\in J_i}D_{g,\omega}$ is a connected interval. By our choice of $h$, it is not difficult to see that for each $\omega \in J_i$,
$$
D_{h,\omega} = D_{g,\omega} + \frac{|J_i| - |\omega - \omega_i|}{|J_i|} \epsilon.
$$
Thus we have
$$
\bigcup_{\omega\in J_i}D_{h,\omega}=\left(\bigcup_{\omega\in J_i}D_{g,\omega}\right)\cup[d_i,d_i+\varepsilon].
$$
In fact, there is a similar extension of all bands of $\bigcup_{\omega\in J_i}\Sigma_{g,\omega}$. Thus, passing from $g$ to $h$, $G_i$ closes up and there is no new gap. Similarly, all gaps in $I$ close up. Thus we have
$$
I \subset \Sigma_{\beta,h},
$$
which implies that $h$ is exactly what we want.
\end{proof}

Note that in Lemma~\ref{l.5}, $dk_{\beta,g}$ and $dk_{\beta,h}$ are close in the weak-$*$ topology if we choose $r$ small. This is because $dk = \int_\T dk_{\omega} \, d\omega$ and the perturbation only occurs on intervals $\CI := \bigcup^{l}_{i=1} \bigcup^{q-1}_{n=0} (J_i + n\frac{p}{q})$, whose size is bounded by $r$ and which is invariant under shift by $\beta$. Here $dk_\omega$ is the spectral measure of $H_\omega$ and $\delta_0$. Indeed, for any real valued continuous function vanishing at infinity, say $\psi \in C_0(\R)$, we have
\begin{align*}
\left| \int_\R \psi \, dk_{\beta,g} - \int_\R \psi \, dk_{\beta,h} \right| & = \left| \int_\T \left( \int_\R \psi \, dk_{g,\omega} - \int_\R \psi \, dk_{h,\omega} \right) \, d\omega \right| \\
& = \left| \int_\CI \left( \int_\R \psi \, dk_{g,\omega} - \int_\CI \psi \, dk_{h,\omega} \right) \, d\omega \right| \\
& \le \int_\CI \left| \int_\R \psi \, dk_{g,\omega} - \int_\R \psi \, dk_{h,\omega} \right| d\omega \\
& \le \int_\CI 2\|\psi \|_\infty \, d\omega \\
& \le 2r \, \|\psi \|_\infty,
\end{align*}
which, choosing $r$ suitably, can be made arbitrarily small.

Now we are ready to show the following theorem:

\begin{theorem}\label{t.2}
There is dense subset $\mathcal D \subset \T \times C(\T,\R)$ such that for each $(\alpha, f) \in \mathcal D$, $dk_{\alpha, f}$ is singular and $\Sigma_{\alpha,f}$ contains an interval. In particular, all such $\alpha$'s are irrational.
\end{theorem}

\begin{proof}
It suffices to show that for any given $(\alpha,f) \in \T \times C(\T,\R)$ with $\alpha$ rational, any $\delta > 0$ and any (non-degenerate) interval $I \subset \Sigma_{\alpha,f}$, we can find $(\beta,g) \in B_\delta(\alpha,f)$ such that $dk_{\beta,g}$ is singular and $I \subset \Sigma_{\beta,g}$.

First note that by Lemma~\ref{l.3}, the set
$$
D_n := \left\{ (\alpha,f) \in \T \times C(\T,\R) : dk_{\alpha,f}(\mathcal B) > 1 - \frac{1}{n} {\rm\ for\ some }\ \mathcal B {\rm\ with\ } {\rm Leb}(\mathcal B) < \frac{1}{n} \right\}
$$
is open in $\T \times C(\T,\R)$ with respect to the natural metric.

Now let $(\alpha,f) = (\alpha_1,f_1)$ with $\alpha_1 \in \Q$ and $\delta = \delta_1 > 0$ be given. Then by Corollary~\ref{c.1} and Lemma~\ref{l.2}, we can pick $(\alpha',f') \in B_\frac{\delta_1}{4} (\alpha_1,f_1)$ such that $dk_{\alpha', f'}$ is singular and $d(\Sigma_{\alpha_1,f_1}, \Sigma_{\alpha',f'}) < \frac{\delta_1}{4}$. Then we can choose a rational $\alpha_2$ such that
$$
|\alpha_2 - \alpha'| < \frac{\delta_1}{4}, \quad d(\Sigma_{\alpha_1,f_1}, \Sigma_{\alpha_2,f'}) < \frac{\delta_1}{4}\ {\rm and }\ (\alpha_2,f') \in D_2.
$$
Then by Lemma~\ref{l.5} and the discussion following it, if we choose $r$ sufficiently small, we can perturb $f'$ to $f_2$ such that $I \subset \Sigma_{\alpha_2,f_2}$ and
$$
(\alpha_2,f_2) \in B_\frac{\delta_1}{2}(\alpha_1,f_1) \cap D_2.
$$

Thus we can find $0 < \delta_2 < \frac{\delta_1}{2}$ such that
$$
\overline B_{\delta_2}(\alpha_2,f_2) \subset B_\frac{\delta_1}{2}(\alpha_1,f_1) \cap D_2.
$$

Now repeat this procedure, and by induction we can find a sequence $\{(\alpha_k, f_k)\}_{k \ge 1} \subset (\Q \cap \T) \times C(\T,\R)$ and a sequence of positive numbers $\{\delta_k\}_{k \ge 1}$ such that
$$
\overline B_{\delta_{k+1}}(\alpha_{k+1},f_{k+1}) \subset B_{\delta_k}(\alpha_k,f_k) \cap D_{k+1}, \ I \subset \Sigma_{\alpha_k,f_k}
$$
with $\delta_k < \frac{\delta_{k-1}}{2}$.

Thus there is $(\alpha_\infty, f_\infty) \in \T \times C(\T,\R)$ such that
$$
\lim_{k \to \infty} (\alpha_k, f_k) = (\alpha_\infty, f_\infty) \in B_\delta(\alpha,f).
$$
Since $\{(\alpha_k, f_k),\ k\ge n\}$ is contained in a ball whose closure is contained in $D_n$, we have
$$
(\alpha_\infty, f_\infty) \in D_n, \quad \forall n \ge 1.
$$
Thus $dk_{\alpha_\infty,f_\infty}$ is singular and $\alpha_\infty$ is irrational. Furthermore, by Lemma~\ref{l.2}, we also have $I \subset \Sigma_{\alpha_\infty,f_\infty}$. This completes the proof of Theorem~\ref{t.2}.
\end{proof}

Theorem~\ref{t.2} shows that zero-measure spectrum is indeed a strictly stronger property than singularity of the density of states measure. In addition, note that the examples with positive-measure spectrum exhibited in Theorem~\ref{t.2} actually have intervals in their spectrum and hence the interior of the spectrum is non-empty. This ``failure of Cantor spectrum'' is a new phenomenon. Indeed, to the best of our knowledge, no (aperiodic) one-frequency quasi-periodic Schr\"odinger operators were previously known for which the spectrum is not nowhere dense.

\bigskip

To conclude, we ask the following question, which motivated us to prove the results presented in the present section: Is it true that for every $\alpha \not\in \Q$, the set
$$
\{ f \in C(\T,\R) : \mathrm{Leb}(\Sigma) = 0 \}
$$
is a dense $G_\delta$?

\end{document}